\documentclass[10pt,a4paper]{article}
\usepackage[margin=2.5cm]{geometry}

\usepackage{graphicx,xcolor}
\definecolor{newdarkblue}{RGB}{0,0,100}
\usepackage[colorlinks=true,hyperindex=true,allcolors=newdarkblue]{hyperref}

\usepackage{amsmath,amsfonts,amssymb}
\usepackage{esint,braket,mathtools}

\usepackage[shortlabels]{enumitem}
\usepackage{nameref}
    \makeatletter
    \def\namedlabel#1#2{\begingroup
    \def\@currentlabel{#2}%
    \label{#1}\endgroup
    }
    \makeatother

\usepackage{amsthm}
\theoremstyle{plain}
    \newtheorem{theorem}{Theorem}[section]
    \newtheorem{proposition}[theorem]{Proposition}
    \newtheorem{lemma}[theorem]{Lemma}
    \newtheorem{corollary}[theorem]{Corollary}
    \newtheorem*{claim*}{Claim}
    
\theoremstyle{remark}
    \newtheorem{remark}[theorem]{Remark}
    \newtheorem{example}[theorem]{Example}
    \newtheorem{definition}[theorem]{Definition}

\newcommand{\rr}{{\mathbf R}}
\newcommand{\zz}{{\mathbf Z}}
\newcommand{\nn}{{\mathbf N}}

\newcommand{\cA}{{\cal A}}

\renewcommand{\P}{\mathbf{P}}
\newcommand{\Q}{\mathbf{Q}}
\newcommand{\bmu}{\boldsymbol{\mu}}
\DeclareMathOperator{\supp}{supp}

\newcommand{\Sc}{h_{\textnormal{c}}}
\newcommand{\Sr}{h_{\textnormal{r}}}

	\makeatletter
	\providecommand*{\diff}%
	{\@ifnextchar^{\DIfF}{\DIfF^{}}}
	\def\DIfF^#1{%
	\mathop{\mathrm{\mathstrut d}}%
	\nolimits^{#1}\gobblespace}
	\def\gobblespace{%
	\futurelet\diffarg\opspace}
	\def\opspace{%
	\let\DiffSpace\!%
	\ifx\diffarg(%
	\let\DiffSpace\relax
	\else
	\ifx\diffarg%
	\let\DiffSpace\relax
	\else
	\ifx\diffarg\{%
	\let\DiffSpace\relax
	\fi\fi\fi\DiffSpace}
    \newcommand{\dd}{\diff}

\newcommand{\Exp}[1]{\mathrm{e}^{#1}}

\begin{document}

\title{On the Ziv--Merhav theorem beyond Markovianity II: \\ {leveraging the thermodynamic formalism}}
\author{N.\ Barnfield\textsuperscript{a}, R.\ Grondin\textsuperscript{a}, G.\ Pozzoli\textsuperscript{b} and R.\ Raqu\'epas\textsuperscript{c}}
\date{}

\maketitle

\begin{center}
  \begin{minipage}[b]{0.5\textwidth}
    \small
    \centering
    a. McGill University \\
    Department of Mathematics and Statistics \\
    Montr\'{e}al QC, Canada
  \end{minipage}%
  \begin{minipage}[b]{0.5\textwidth}
    \small
    \centering
    b. CY Cergy Paris Universit\'e\\
    Department of Mathematics\\
    Cergy-Pontoise, France
  \end{minipage}  
  
  \vspace{1em}
  
  \begin{minipage}[b]{0.45\textwidth}
    \small
    \centering
    c. New York University \\
    Courant Institute of Mathematical Sciences \\
    New York NY, United States\\
  \end{minipage}
\end{center}

\begin{abstract}
    We prove asymptotic results for a modification of the cross-entropy estimator originally introduced by Ziv and Merhav in the Markovian setting in 1993. Our results concern a more general class of decoupled measures on {shift spaces over a finite alphabet} and in particular imply strong asymptotic consistency of the modified estimator for all pairs of 
    functions of stationary, irreducible, finite-state Markov chains satisfying a mild decay condition.
    Our approach is based on the study of a rescaled cumulant-generating function called the cross-entropic pressure, importing to information theory some techniques from the study of large deviations within the thermodynamic formalism.


\end{abstract}

% \tableofcontents

\section{Introduction}

Entropy, cross entropy and relative entropy are key quantities in information theory and many adjacent fields: statistical physics, dynamical systems, pattern recognition, etc. The problem of estimating them from data, making as few assumptions as possible on the sources generating the data, has a long history. 
In 1993, Ziv and Merhav proposed, based on the seminal work~\cite{LZ78}, a procedure for estimating the (specific) cross entropy 
\begin{equation}
\label{eq:def-cross}
    \Sc(\Q|\P)\coloneqq \lim_{n\to\infty}-\frac{1}{n}\sum_{a\in\mathcal{A}^n}\Q[a]\ln \P[a]
\end{equation}
for two ergodic measures~$\P$ and~$\Q$ on a space of sequences with values in some finite alphabet~$\cA$.\footnote{Existence of the limit is discussed in the next section.}
Roughly summarized, their procedure entails counting the number $c_N^{\textnormal{ZM}}(y|x)$ of words in a sequential parsing of~$y_1^N$ using the longest possible strings that can be found in~$x_1^N$, where $y$ and $x$ are produced respectively by the sources $\Q$ and $\P$,
and then computing 
\begin{equation}
\label{eq:original-ZM-est}
    Q_N^{\textnormal{ZM}}(y,x) \coloneqq \frac{c_N^{\textnormal{ZM}}(y|x) \ln N}{N}
\end{equation}
as an estimator for~$\Sc(\Q|\P)$; see~\cite[\S{II}]{MZ93} {and Example~\ref{ex:parsings}} below. {This provided an operational point of view on cross entropy (and thus on relative entropy, as explained below) that was not in the literature on statistical mechanics or dynamical systems.} They have shown that if $\P$ and $\Q$ are ergodic Markov measures, then this estimator is strongly asymptotically consistent in the sense that 
\begin{equation}
\label{eq:as-conv-1}
    \lim_{N\to\infty} Q_N^{\textnormal{ZM}}(y,x) =  \Sc(\Q|\P)
\end{equation}
for $(\P\otimes\Q)$-almost every~$(x,y)$. As a particular case, if $\P=\Q$ then this suggests an entropy estimator.
In a recent work, we have extended this almost sure result to a class of decoupled measures that includes regular g-measures and 1-dimensional Gibbs measures for summable interactions satisfying very mild assumptions, but not all hidden-Markov models with finite hidden alphabet~\cite{BGPR}. This seemed to be the best one could do by following the parsing procedure in the spirit of the original proof.

In this sequel, we propose a modification of Ziv and Merhav's estimator, denoted $Q_N$, for which we can prove {considerably more general} results relying on the properties of the \emph{cross-entropic pressure}
\begin{equation}
\label{eq:def-pres}
    \overline{q}(\alpha) \coloneqq \limsup_{n\to\infty} \frac 1n \ln \sum_{a \in \supp \P_n} \Exp{-\alpha\ln\P[a]} \Q[a].
\end{equation}
{This approach is inspired by the thermodynamic formalism and statistical mechanics of lattice gases. Notably, it does not involve the construction of any auxiliary parsing; cf.~\cite{MZ93,BGPR}.}
One of our main results, Theorem~\ref{thm:sc-gen}, confines in the almost sure sense the limit points of $Q_N$ to the subdifferential of $\overline{q}$ at the origin, as soon as the shift-invariant measures $\P$ and $\Q$ both satisfy variants of the upper- and selective lower-decoupling assumptions of~\cite{CJPS19} and some mild nondegeneracy assumption; see Conditions~\ref{it:UD},~\ref{it:SLD}  and~\ref{it:ND} below. 
Early numerical experiments using hidden-Markov models suggest that the practical performance is very similar to that of the original ZM estimator.

The rest of the paper is organized as follows. In Section~\ref{sec:setting}, we properly define the objects of interest, state a central theorem, and argue that if $\overline{q}(\alpha)$ is a regular enough limit near $\alpha = 0$, then this theorem immediately implies almost sure convergence to the cross entropy, as in~\eqref{eq:as-conv-1}. This central result is proved in Section~\ref{sec:proof}, building on ideas from~\cite{Ko98,CJPS19,CDEJR23w,CR23}. In Section~\ref{sec:open-prob}, we state and discuss an open problem concerning the passage to~\eqref{eq:as-conv-1} in the absence of regularity of~$\bar{q}$. It is framed as a ``rigidity'' problem for the cross-entropy analogue of the Shannon--McMillan--Breiman theorem, and we show that partial results on this problem yield strong asymptotic consistency in special cases. Finally, important examples, including hidden-Markov measures and $\psi$-mixing measures, are discussed in Section~\ref{sec:Examples}.

\paragraph*{Acknowledgments}
    The authors would like to thank G.\ Cristadoro and V.\ Jak\v{s}i\'{c} for stimulating discussions on the topic of this article. {We are particularly grateful to N.\ Cuneo for sharing personal notes on generalizations of~\cite[\S{4.1}]{BJPP18}, which allowed us to prove Lemma~\ref{lem:q-exists} and Theorem~\ref{thm:left-der}.} 

    The research of NB and RR was partially funded by the \emph{Fonds de recherche du Qu\'ebec\,---\,Nature et technologies} (FRQNT) and by the Natural Sciences and Engineering Research Council of Canada (NSERC). The research of RG was partially funded by the {Rubin Gruber Science Undergraduate Research Award} and {Axel W Hundemer}. The research of GP was supported by the CY Initiative of Excellence through the grant Investissements d'Avenir ANR-16-IDEX-0008, and was done under the auspices of the \emph{Gruppo Nazionale di Fisica Matematica} (GNFM) section of the \emph{Istituto Nazionale di Alta Matematica} (INdAM) while GP was a post-doctoral researcher at University of Milano-Bicocca (Milan, Italy). Part of this work was done during a stay of the four authors in Neuville-sur-Oise, funded by CY Initiative (grant \emph{Investissements d'avenir} ANR-16-IDEX-0008).

\section{Setting and main result}
\label{sec:setting}

Let $\cA$ be a finite set and let $T : \cA^\nn \to \cA^\nn$ be the left shift. We use $[a] \coloneqq  \{x \in \cA^\nn : x_1^n = a\}$ for basic cylinder sets, where $x_1^n$ is used for the symbols $x_1x_2 \dotsc x_n$ in the sequence $x = (x_k)_{k=1}^\infty$. Let us introduce the \emph{waiting times}
\[ 
    W_\ell(y,x) \coloneqq  \inf\{r \in \nn : x_{r}^{r+\ell - 1} = y_1^\ell\},
\]
and \emph{match lengths}
\[ 
    \Lambda_N(y,x) \coloneqq   \sup\{\ell\in\nn : W_\ell(y,x) \leq N-\ell+1 \}.
\]
Their relation to entropic quantities has been pioneered in~\cite{WZ89} and progressively refined in~\cite{Sh93,Ko98,CDEJR23w,CR23}.

\begin{definition}\label{def:mZM}
 The \emph{modified Ziv--Merhav (mZM) parsing} of $y_1^N$ with respect to~$x_1^N$ is defined sequentially as follows: 
    \begin{itemize}
        \item The \emph{first word} $\overline{w}^{(1,N)}(y,x)$ is the shortest prefix of $y_1^N$ that does not appear in $x_1^N$, that is 
        \[
            \overline{w}^{(1,N)}(y,x) \coloneqq  y_1^{\min\{N,\Lambda_N(y,x)+1\}}
        \]
        or all of $y_{1}^N$ if no such prefix exists\,---\,which would terminate the procedure.
        \item Given that the first~$i$ words collected so far have lengths summing to~$L_{i,N} < N$, the \emph{next word} $\overline{w}^{(i+1,N)}(y,x)$ is the shortest prefix of~$y_{L_i+1}^N$ that does not appear in $x_1^N$, that is 
        \[
            \overline{w}^{(i+1,N)}(y,x) \coloneqq  y_{L_i+1}^{\min\{N,L_i + \Lambda_N(T^{L_{i}}y,x)+1\}},
        \]
        or all of $y_{L_i+1}^N$ if no such prefix exists\,---\,which would terminate the procedure.
    \end{itemize}
    The \emph{mZM estimator} is then given by
    \[ 
        Q_N(y,x) \coloneqq  \frac{c_N(y|x) \ln N}{N - c_N(y|x)},
    \]
    where $c_N(y|x)$ is the number of words in the mZM parsing of~$y_1^N$ with respect to~$x_1^N$.
\end{definition}

{
\begin{example}
\label{ex:parsings}
    Consider the following two sequences with values in~$\cA=\{\mathsf{0},\mathsf{1}\}$:
    \begin{align*}
        x &= \mathsf{1010010101001011101010011}\dotsc, \\
        y &= \mathsf{0101100101100101010100110}\dotsc
    \end{align*}
    In the original ZM parsing, the first word is the longest prefix of $y_1^{25}$ that appears somewhere in~$x_1^{25}$, which is $\mathsf{01011}$, and the second word is obtained in the same way after having removed $\mathsf{01011}$ at the beginning of $y_1^{25}$, and so on and so forth until 
    \[ 
        y_1^{25} = \mathsf{01011|001011|00101010|10011|0}.
    \]
    The estimated cross entropy between the sources is 
    \[ 
        Q^{\textnormal{ZM}}_{25}(y,x) = \frac{5 \ln 25}{25} \approx 0.64.
    \]
    In the mZM parsing, the first word is the shortest prefix of $y_1^{25}$ that does not appear anywhere in~$x_1^{25}$, which is $\mathsf{010110}$, and the second word is obtained in the same way after having removed $\mathsf{010110}$ (i.e.\ after removing one more letter than in the ZM parsing) at the beginning of $y_1^{25}$, and so on and so forth until 
     \[ 
        y_1^{25} = \mathsf{010110|010110|01010101|00110}.
    \]
    The estimated cross entropy between the sources is 
    \[ 
        Q_{25}(y,x) = \frac{4 \ln 25}{25-4} \approx 0.61.
    \]
\end{example}
}

At finite~$N$, our estimator differs from the original ZM estimator in two ways. First, we are counting words that consist of substrings of~$y_1^N$ that cannot be found in~$x_1^N$ instead of substrings of~$y_1^N$ that can be found in~$x_1^N$. Second, we are dividing by $N-c_N(y|x)$ instead of~$N$.\footnote{In cases $c_N(y|x) = N$, i.e.\ in cases where we are dividing by~$0$, no bigram in $x_1^N$ can be found in~$y_1^N$ and it is reasonable to estimate the cross entropy to be infinite.} Each of these modifications has two motivations. 

The change from longest substring that can be found to shortest substring that cannot be found mimics more closely a Lempel--Ziv parsing on the concatenation $x_1^N y_1^{n}$ that is sometimes used in practice as a variant of the ZM algorithm \cite{KPK01,BCL02,CF05,B+08}. Also, working with this variant bypasses some technical estimates that are essential to Ziv and Merhav's argument, and which may fail beyond the scope of our previous work~\cite{BGPR}; see Section~3.4 there.

If one believes that the original ZM estimator is close to being unbiased, then the change from $N$ to $N-c_N(y|x)$ can be seen as an attempt to correct the bias introduced by our first change. Indeed, the~$N$ in the denominator in~\eqref{eq:original-ZM-est} is\,---\,as opposed to the $\ln N$ in the numerator\,---\,tailored to the length of~$y_1^N$, so we are correcting for the fact that, from the ZM point of view, our first change drops the ``effective length'' of~$y_1^N$ from $N$ to something closer to $N-c_N(y|x)$. {In Example~\ref{ex:parsings}, if we had divided by $25$ instead of $25-4$, then we would have obtained an estimation of approximately $0.52$ instead of $0.61$, which would be further from the original ZM estimation of $0.64$.} It will appear as a consequence of an intermediate result that $c_N = o(N)$ under our assumptions, so that this correction is irrelevant in the limit; see Remark~\ref{rem:cN-is-oN}. This arguably \emph{ad hoc} correction to the denominator has the following two advantages: it naturally appears in some technical arguments and seems to improve the estimation at finite~$N$; see Section~\ref{ssec:iii-iv} (especially  Remark~\ref{rmk:corr}) and Section~\ref{ssec:hmm}.

Throughout this work, two stationary sources of strings of symbols from~$\cA$ are described by shift-invariant measures\footnote{Throughout this article, all measure-theoretic notions are considered with respect to the $\sigma$-algebra generated by basic cylinders, or equivalently with respect to the Borel $\sigma$-algebra for the product topology built from the discrete topology on~$\cA$.} $\P$ and~$\Q$ on $\cA^\nn$, and we use
\begin{align*}
     \supp \P_n &\coloneqq  \{a \in \cA^n : \P[a] > 0\}, 
     &
     \supp \P &\coloneqq  \{x \in \cA^\nn : \P[x_1^n] > 0 \text{ for all } n \}, \\
     \supp \Q_n &\coloneqq  \{a \in \cA^n : \Q[a] > 0\}, 
     &
     \supp \Q &\coloneqq  \{x \in \cA^\nn : \Q[x_1^n] > 0 \text{ for all } n \}.
\end{align*}
We only consider the case where the sources are independent from one another, i.e.\ samples from the product measure~$\P \otimes \Q$ on~$\cA^\nn \times \cA^\nn$.

We recall that the (specific) \emph{cross entropy} of~$\Q$ with respect to~$\P$, i.e.\ the limit~\eqref{eq:def-cross}, need not exist in full generality, but in cases where it exists, the (specific) \emph{relative entropy} $\Sr(\Q|\P)$ of $\Q$ with respect to~$\P$ can be decomposed\,---\,or even defined\,---\,as
\[
    \Sr(\Q|\P) := \Sc(\Q|\P)-h(\Q), 
\]
where the (specific) \emph{entropy} $h(\P)$ of~$\P$ is the limit
\[
    h(\P)\coloneqq \lim_{n\to\infty}-\frac{1}{n}\sum_{a\in \mathcal{A}^n}\P[a]\ln \P[a].
\]
This last limit is guaranteed to exist in~$[0, \ln (\#\cA)]$ by subadditivity of entropy for measures on products of finite sets and Fekete's lemma. {While the cross entropy naturally  appears e.g.\ in discussions of the Kraft--McMillan inequality and optimal code length~\cite[\S{5.4}]{CoTh}, one could argue that the relative entropy appears directly in more fundamental applications. But since estimation of the entropy term in the decomposition of the relative entropy is by now well understood, there is no serious harm in focusing on the cross entropy term.}

As in~\cite{CJPS19,BCJPPExamples,CDEJR23w,CR23}, the following two decoupling conditions are at the heart of our arguments.
\begin{description}
    \item[{UD}\namedlabel{it:UD}{UD}]
    {A measure $\P$ is said to be \emph{upper decoupled} if there exist
    a nonnegative, $o(n)$-sequence $(k_n)_{n=1}^\infty$ and an integer constant $\tau\geq 0$
    with the following property: for all~$n\in\nn$, $a \in \supp \P_n$, $\xi\in\supp\P_\tau$, $m\in\nn$ and $b \in \supp \P_m$,
    \begin{equation}
        {\P[a \xi b]} \leq \Exp{k_n}{\P[a]\P[b]}.
    \end{equation}}
    \item[{SLD}\namedlabel{it:SLD}{SLD}]
    {A measure $\P$ is said to be \emph{selectively lower decoupled} if there exist a nonnegative, $o(n)$-sequence $(k_n)_{n=1}^\infty$ and an integer constant $\tau\geq 0$  
     with the following property: for all~$n\in\nn$, $a \in \supp \P_n$, $m\in\nn$ and $b \in \supp \P_m$, there exists $0\leq\ell=\ell(a,b)\leq\tau$ and $\xi=\xi(a,b) \in \supp\P_\ell$ such that 
     \begin{equation}
         {\P[a \xi b]}  \geq \Exp{-k_n} {\P[a]\P[b]}.
     \end{equation}}
\end{description}

\begin{remark}
\label{rem:increase-params}
    By taking for each $n$ the maximum of the first $n$ terms, one can assume without loss of generality that the sequence $(k_n)_{n\in\nn}$ is nondecreasing. 
    Furthermore, by taking the index-wise maximum, one can consider the same sequence when assuming both conditions for several measures. One can also show that, both in~\ref{it:UD} and~\ref{it:SLD}, one can always replace $\tau$ with some $\tau'\geq \tau$, and thus consider a common~$\tau$ when assuming both conditions for several measures; see~\cite[\S{2.2}]{CJPS19}.
\end{remark}

\begin{remark}
\label{rem: tau_n}
    The reader familiar with~\cite{CJPS19,CDEJR23w,CR23} will notice the following difference: throughout the present work, the ``gap size'' $\tau$ in~\ref{it:UD} and~\ref{it:SLD} is a constant that is not allowed to grow with~$n$ (the length of the first string~$a$). Adapting the proofs of the following results to the case $\tau_n = o(n)$ requires only minor changes, if any: Lemma~\ref{lem:a-priori-bounds-a}, Lemma~\ref{lem:cross-SMB-gap}, Lemma~\ref{lem:q-exists}, Theorem~\ref{thm:abstract-modified-parsing} and Proposition~\ref{prop:diff-minus}. However, as is, the proofs of Proposition~\ref{prop:diff-plus} and Theorem~\ref{thm:left-der} \emph{do not} allow for such generalizations. 
    
    To salvage the conclusion of Proposition~\ref{prop:diff-plus} (and thus of Theorem~\ref{thm:sc-gen}) in the case~$\tau_n = o(n)$ one can make the extra assumption that $\P$ satisfies~\ref{it:SE} in order to leverage the improved bounds from Remark~\ref{rem:uniform-if-SE}. As for Theorem~\ref{thm:left-der}, allowing $\tau_n = o(n)$ for the~\ref{it:UD} and~\ref{it:SLD} properties of~$\P$ causes no problem, but the use of $\tau_n = O(1)$ for the~\ref{it:UD} property of~$\Q$ is crucial to the argument.
\end{remark}

Recall that we have introduced the cross-entropic pressure $\overline{q} = \limsup_{n\to\infty} \tfrac 1n q_n$ where
\begin{equation*}
     q_n(\alpha)\coloneqq  \ln \sum_{a \in \supp \P_n} \Exp{-\alpha \ln \P_n[a]} \Q[a].
\end{equation*}
Note that this is not the rescaled cumulant-generating function for the estimator~$Q_N$, but rather for the sequence of logarithmic probabilities in the cross-entropy analogue of the Shannon--McMillan--Breiman theorem discussed in Section~\ref{ssec:proof-prelim}\,---\,sometimes called the Moy--Perez theorem~\cite{Mo61,Ki74,Or85,Ba85}.

We will use the notation~$q(\alpha)$ for the limit when it exists. Note that $\overline{q}$ is a nondecreasing, (possibly improper) convex function as a limit superior of functions that are nondecreasing, and convex by H\"older's inequality; the same is true for~$q$ when it exists. It is also easy to see that $q(0)$ exists and equals~$0$. Relations between this cross-entropic pressure and the cumulant generating functions of waiting times have been studied extensively in~\cite{AACG22,CR23}.

We will denote by $[D_- \overline{q}(0),D_+ \overline{q}(0)]$ the interval delimited by the (possibly infinite) left and right derivatives of~$\overline{q}$ at the origin.
The following is a relatively mild nondegeneracy condition.
\begin{description}
    \item[{ND}\namedlabel{it:ND}{ND}] A pair $(\P,\Q)$ is said to be \emph{nondegenerate} if the limit superior ${\overline{q}(-1)}$
is negative. 
\end{description}
\begin{remark}
\label{rem:ND-decay}
    Using monotonicity and convexity, Condition~\ref{it:ND} can be reformulated as the requirement that $\overline{q}$ is not identically vanishing to the left of the origin.
    It is straightforward to show that~\ref{it:ND} will hold if there exists $\gamma_+ < 0$ such that, for all $n\in\nn$ large enough, either $\max_{a \in \supp \Q_n} \P[a] \leq \Exp{\gamma_+ n}$ or $\max_{a \in \supp \P_n} \Q[a] \leq \Exp{\gamma_+ n}$, i.e.\ if one of the two measures decays (at least) exponentially fast on the support of the other. We will come back on this point in the context of examples in Section~\ref{sec:Examples}. We will further discuss~\ref{it:ND} in Remark~\ref{rem:ND-left-der}.
\end{remark}

These conditions serve as the hypotheses for a central theorem, Theorem~\ref{thm:sc-gen} below. In fact, most of what remains of the present work is dedicated to proving, discussing, refining and finding alternatives to different building blocks of this theorem. While we use this theorem as a guiding thread, it should be noted that, in some instances, the refinements and alternatives (such as Theorem~\ref{thm:abstract-modified-parsing} and Corollaries~\ref{cor:UD-tau-0} and~\ref{cor:wG}) could be of greater or equal interest.

\begin{theorem}
\label{thm:sc-gen}
    Suppose that $\P$ satisfies~\ref{it:SLD} and~\ref{it:UD}, that $\Q$ satisfies~\ref{it:UD}, and that~\ref{it:ND} holds. If, in addition $\supp \Q \subseteq \supp \P$, then 
    \begin{subequations}
    \label{eq:subdiff-bounds}
    \begin{alignat}{1}
        D_- \overline{q}(0)
        &\leq \liminf_{N\to\infty} Q_N(y,x)  \label{eq:subdiff-lb}
        \\
        &\leq \limsup_{N\to\infty} Q_N(y,x)\leq D_+\overline{q}(0)
        \label{eq:subdiff-ub}
    \end{alignat}
    \end{subequations}
    for $(\P\otimes\Q)$-almost every~$(x,y)$.
\end{theorem}

\begin{remark}
    We are making an assumption on how the supports of~$\P$ and~$\Q$ relate to each other. This assumption amounts to a relation of absolute continuity between the marginals, and will be enforced throughout the rest of the article. However, note that a straightforward argument using Birkhoff's ergodic theorem shows that if $\supp \Q \not\subseteq \supp \P$ and $\Q$ is ergodic, then $Q_N(y,x) \to \infty$ as $N \to \infty$ for $(\P\otimes\Q)$-almost every~$(x,y)$. This can be interpreted as consistent with~\eqref{eq:subdiff-bounds} (and also with~\eqref{eq:when-diff} below), provided that one uses the appropriate arithmetic conventions for $0^{-\alpha}$, $\ln 0$, $\ln \infty$, $0 \pm \infty$ and so on.
\end{remark}

\begin{remark}
\label{rem:ND-left-der}
    We are also assuming that~\ref{it:ND} holds. Note that~\ref{it:ND} can only fail if $D_-\overline{q}(0) = 0$, in which case~\eqref{eq:subdiff-lb} trivially holds. Furthermore, we will see in Section~\ref{ssec:left-deriv} that, under some additional conditions, $D_-\overline{q}(0) = 0$ in turn implies that $\Sc(\Q|\P) = 0$, and hence that both $h(\Q) = 0$ and $\Sr(\Q|\P) = 0$, which is a particularly degenerate case.
\end{remark}

The strength of the conclusions that one can directly draw from Theorem~\ref{thm:sc-gen} will depend on further regularity properties of the pressure at the origin, as made explicit in the next corollary. We emphasize that, under the above decoupling conditions, such properties are by no means guaranteed; see Section~\ref{sec:Examples} and~\cite{CJPS19,BCJPPExamples,CR23}.
However, they hold for measures that fit within the classical uniqueness regimes of the thermodynamic formalism; see e.g.~\cite[\S{4.3}]{Ke98} and~\cite[\S{4}]{Wa01}.

\begin{corollary}
\label{cor:when-diff}
    Suppose, in addition to the hypotheses of Theorem~\ref{thm:sc-gen}, that $q$ exists as a finite limit on an open interval containing the origin, and that $q$ is differentiable at the origin. Then,
    \begin{equation}
    \label{eq:when-diff}
        \lim_{N\to\infty} Q_N(y,x)  
        \\
        = \Sc(\Q|\P)
    \end{equation}
    for $(\P\otimes\Q)$-almost every~$(x,y)$.
\end{corollary}

\begin{proof}
    Since $\P$ satisfies~\ref{it:UD}, we have 
    \[ 
        \lim_{n\to\infty}\frac{-\ln \P[y_1^n]}{n} = h_{\P}(y)
    \]
    for $\Q$-almost every~$y$, where $h_{\P}$ is a nonnegative, measurable, shift-invariant function integrating to~$\Sc(\Q|\P)$ with respect to~$\Q$; see Lemma~\ref{lem:cross-SMB-gap} below. 
    On the other hand,
    if $q$ exists as a finite limit on an open interval containing the origin and is differentiable at the origin, then a standard large-deviation argument~\cite[\S{II.6}]{Ell} yields
    \[ 
        \lim_{n\to\infty}\frac{-\ln \P[y_1^n]}{n} = q'(0)
    \]
    for $\Q$-almost every~$y$.
    Hence, we may use $D_-\bar{q}(0) = D_+\bar{q}(0) = \Sc(\Q|\P)$ in~\eqref{eq:subdiff-bounds} to deduce~\eqref{eq:when-diff}.
\end{proof}

In some sense, the gap to bridge from~\eqref{eq:subdiff-bounds} to~\eqref{eq:when-diff} outside of this regular situation is not a problem regarding the behavior of the mZM estimator itself as much as it is a problem regarding a certain rigidity of the convergence in an analogue of the Shannon--McMillan--Breiman theorem for cross entropies; this will be discussed in Section~\ref{sec:open-prob}.

\section{Proof of Theorem~\ref{thm:sc-gen}}
\label{sec:proof}

Theorem~\ref{thm:sc-gen} and Corollary~\ref{cor:when-diff} will follow from Theorem~\ref{thm:abstract-modified-parsing} and Propositions~\ref{prop:diff-minus} and~\ref{prop:diff-plus} below. 
{Before we proceed with their proofs, let us provide some intuition and} introduce some notation. If $\overline{w}^{(i,N)}(y,x)$ (sometimes simply $\overline{w}^{(i,N)}$) is a word coming from the mZM parsing of~$y_1^N$ with respect to~$x_1^N$, we write $\ell_{i,N}$ for its length. Furthermore, we denote by $\underline{w}^{(i,N)}(y,x)$ (sometimes simply $\underline{w}^{(i,N)}$) the string obtained by considering $\overline{w}^{(i,N)}$ without its last letter, and by $\underline \ell_{i,N}$ its length. Sometimes, we suppress the $N$-dependence from the notation and simply write $\overline w^{(i)}$, $\ell_i$, $\underline w^{(i)}$ and $\underline \ell_i$ for readability. 
Our treatment of the estimator requires the following 4 ingredients with high enough probability:
\begin{enumerate}
    \item[I.] \emph{The number $c_N$ of words in the mZM parsing can be bounded above in terms of the probabilities $\P[\overline{w}^{(i,N)}]$.} 

    To see why this is plausible, think of $W_{\ell}(y,x)$ as a geometric random variable with parameter $\P[y_1^\ell]$ as in e.g.~\cite{GS97,CR23}: $W_{\ell}(y,x)$ is very likely to be less than $N-\ell+1$ as long as $\P[y_1^\ell]$ stays above $N^{-1+\epsilon}$. This suggests that, most likely, $\ln \P[\overline{w}^{(1,N)}] < -(1-\epsilon)\ln N$. If there is enough decoupling that the same is true not only for $i=1$ but for all~$i$ up to $c_N-1$, then summing over~$i$ suggests that, most likely, 
    $$
        (1-\epsilon)[c_N-1] \ln N 
            \lesssim -\sum_{i=1}^{c_N-1} \ln \P[\overline{w}^{(i,N)}].
    $$

    \item[II.] \emph{The number $c_N$ of words in the mZM parsing can be bounded below in terms of the probabilities $\P[\underline{w}^{(i,N)}]$.}

    To see why this is plausible, consider the same geometric approximation as for Ingredient~I: $W_{\ell}(y,x)$ is very likely to be more than $N$ once $\P[y_1^\ell]$ drops below $N^{-1-\epsilon}$. This suggests that, most likely, $\ln \P[\underline{w}^{(1,N)}] > -(1+\epsilon)\ln N$. If there is enough decoupling that the same is true not only for $i=1$ but for all~$i$ up to $c_N$, then summing over~$i$ suggests that, most likely, 
    $$
        (1+\epsilon)[c_N-1] \ln N 
            \gtrsim -\sum_{i=1}^{c_N} \ln \P[\underline{w}^{(i,N)}].
    $$
    
    \item[III.] \emph{Minus the sum of the logarithms of the probabilities $\P[\overline{w}^{(i,N)}]$ can be bounded above in terms of the cross entropy or the cross-entropic pressure.}

    To see why this is plausible, recall that under suitable decoupling conditions, there is a large-deviation principle for $(-\ln \P[y_1^\ell])_{\ell=1}^\infty$ with $y$ sampled from~$\Q$~\cite{CR23}. Therefore, it should be very unlikely for any of the $-\tfrac{1}{\ell_{i,N}} \ln \P[\overline{w}^{(i,N)}]$ to fall outside the zero-set of the rate function as the words get long. But this zero-set is expected to be the subdifferential of the pressure at~$0$.
    
    \item[IV.] \emph{Minus the sum of the logarithms of the probabilities $\P[\underline{w}^{(i,N)}]$ can be bounded below in terms of the cross entropy or the cross-entropic pressure.}

    The heuristics here mimic those for Ingredient~III and only the details of the proof technique will change.
\end{enumerate}

\subsection{Preliminaries}
\label{ssec:proof-prelim}

Note that the above heuristic arguments rely on the words in the mZM parsing being long enough when~$N$ is large enough. It turns out that, for some technical arguments, we will also need these length not to grow too fast as $N$ increases. This is why the following \emph{a priori} bounds on the length of the words in the mZM parsing will play an important role in establishing every ingredient. For the rest of this section, we will denote this \emph{a priori} upper bound by
\[
    \ell_{+,N}\coloneqq \frac{8\ln N}{-\overline{q}(-1)}.
\]
The assumptions that~$\P$ and~$\Q$ are shift invariant and that $\supp\Q\subseteq\supp\P$ are enforced throughout, but other assumptions will be explicitly imposed when needed.
\begin{lemma}
\label{lem:a-priori-bounds-a}
    Suppose that $\P$ satisfies~\ref{it:SLD} and that~\ref{it:ND} holds. Then, for all $\lambda > 0$, we have, for $N$ large enough, that  with probability at least $1-2N^{-2}$, 
    \begin{equation}
    \label{eq:a-priori-bounds-a}
       \lambda \leq |\underline w^{(i,N)}|<|\overline{w}^{(i,N)}| \leq \ell_{+,N}
    \end{equation}
    for $i < c_N$, and the corresponding upper bound for~$\overline{w}^{(c_N,N)}$.
\end{lemma}
\begin{proof}
    We start with the upper bound, relying on~\ref{it:ND}.
    By a mere union bound, we have 
    \begin{align*}
        (\P\otimes\Q) \left\{ (x,y) : W_{\ell}(y, x) < N\right\}
            &\leq \sum_{a \in \supp \P_\ell } \Q[a] \cdot N \P[a] \\
            &= N \Exp{\ell \cdot \frac{1}{\ell} \ln \sum_{a \in \supp \P_\ell } \Q[a]\P[a]} \\ 
            &\leq N \Exp{\ell \frac{{{\overline{q}(-1)}}}{2}}
    \end{align*}
    for all $\ell$ large enough. In particular, with ${\ell_+} = \frac{8}{-{{\overline{q}(-1)}}} \ln N$, we find 
    \begin{align*}
        (\P\otimes\Q) \left\{ (x,y) : W_{\ell_+ }(y, x) < N\right\}
            &\leq N \Exp{\frac{8}{-{{\overline{q}(-1)}}} \ln N \frac{{{\overline{q}(-1)}}}{2}} \\
            &= N^{-3}
    \end{align*}
    for all~$N$ large enough. Using a union bound and shift invariance, we find the following: with probability at least~$1 - N^{-2}$, no substring of $y$ of length $\ell_+ $ appears in~$x_1^N$.

    Let $\lambda \in \nn$ be arbitrary.
    By the Kontoyiannis argument for measures satisfying~\ref{it:SLD}, whenever we have $a \in \supp \Q_{\lambda}$,\footnote{Recall that we are working under the assumption that $\supp \Q_{\lambda} \subseteq \supp \P_{\lambda}$.} we have 
    \begin{align*}
        \P\{x : W_{\lambda}(a,x) > N - 2\lambda\} &\leq (1 - \Exp{-k_{\lambda}}\P[a])^{\left\lfloor\frac{N - 2\lambda-1}{\lambda+\tau}\right\rfloor}\\
            &\leq\exp\left(-  \frac{N\min_{a\in\supp\Q_\lambda}\P[a]}{\Exp{k_{\lambda}}3\lambda}\right);
    \end{align*}
    see e.g.~\cite[\S{3.1}]{CR23}. 
    Using a union bound over $a \in \supp \Q_\lambda$, we find the following for $N$ large enough: with probability at least $1 - N^{-2}$, all words of length $\lambda$ that are at all likely to appear in~$y_1^N$ do appear in~$x_1^N$.

    Therefore, with probability at least $1-2N^{-2}$, all words in the mZM parsing of~$y_1^N$ have lengths between $\lambda$ and $\ell_+$.
\end{proof}

\begin{remark}
\label{rem:cN-is-oN}
    The bounds~\eqref{eq:a-priori-bounds-a} imply bounds on $c_N(y|x)$. In particular, appealing to the Borel--Cantelli lemma, we have that, $(\P\otimes\Q)$-almost surely, $c_N(y|x) = o(N)$.
\end{remark}

\begin{remark}
    The decay of the probability as $N^{-2}$ is not optimal: a decay as $N^{-\nu}$ starting from some possibly larger $N$ can be obtained by replacing $8$ with $2(2+\nu)$ in the definition of~$\ell_{+,N}$. However, we are ultimately interested in almost sure statements and $N^{-2}$ suffices for the summability required for applications of the Borel--Cantelli lemma. 
\end{remark}

\begin{remark}
\label{rem:uniform-if-SE}
The lower bound can be taken to be growing with $N$ if the $\P$-probabilities of cylinders decay {slowly enough}. For example, one can replace $\lambda$ with some $\ell_{-,N}$ that grows as a power of $\ln N$ under the following mild condition:
\begin{description}
    \item[{SE}\namedlabel{it:SE}{SE}] A measure $\P$ will be said to decay \emph{slowly enough} if there exists $\beta \geq 1$ and $\gamma_- < 0$ such that 
    \begin{equation}
        \min_{a \in \supp \P_n} \P[a] \geq \Exp{\gamma_- n^{\beta}}
    \end{equation}
    for all~$n\in\nn$. 
\end{description}
{To be more precise, the power of the logarithm in $\ell_{-,N}$ can be chosen to be $\beta^{-1}$.}
\end{remark}

{
\begin{remark}
\label{rem:SLD-through-Kon}
    Note that~\ref{it:SLD} is only used through an upper bound on the probability of a waiting time exceeding a certain value, for which we have referred to~\cite[\S{3.1}]{CR23}. This section is based on ideas from~\cite[\S{2}]{Ko98} and~\cite[\S{3}]{CDEJR23w}. In Kontoyiannis original work, a $\psi$-mixing property was used instead of~\ref{it:SLD}, and the former implies the latter. The same remark applies to Theorem~\ref{thm:abstract-modified-parsing}% and Lemma~\ref{lemma:CLD}
    .
\end{remark}

\begin{lemma}
\label{lem:cross-SMB-gap}
    If $\P$ satisfies~\ref{it:UD}, then~$\Sc(\Q|\P)$ exists and
    \[ 
        \lim_{n\to\infty}\frac{-\ln \P[y_1^n]}{n} = h_{\P}(y)
    \]
    for $\Q$-almost every~$y$, where $h_{\P}$ is a nonnegative, measurable, shift-invariant function integrating to~$\Sc(\Q|\P)$ with respect to~$\Q$. If, in addition, $\Q$ is ergodic, then
    \[ 
        \lim_{n\to\infty}\frac{-\ln \P[y_1^n]}{n} = \Sc(\Q|\P)
    \]
    for $\Q$-almost every~$y$. 
\end{lemma}

\begin{proof}
    In view of~\ref{it:UD}, we can apply a gapped version of Kingman's subadditive ergodic theorem~\cite[\S{III}]{Ra23} to the sequence $(f_n)_{n=1}^{\infty}$ of measurable functions on~$(\supp \P, \mathcal{F}, T, \Q)$ defined by $f_n(x) \coloneqq  \ln \P[x_1^n]$.
    This provides all the assertions of the lemma.
\end{proof}

\begin{lemma}
\label{lem:q-exists}
    If $\P$ and $\Q$ both satisfy~\ref{it:UD}, then $q(\alpha)$ exists in $[-\infty,0]$ for all~$\alpha \leq 0$. Moreover, we have the variational representation
    \begin{equation}
    \label{eq:alpha-var-p}
    q(\alpha) = \sup_{\mu \textnormal{ invar}} \int \alpha h_\P - [h_\Q - h_\mu] \dd\mu
    \end{equation}
    for all~$\alpha \leq 0$, where the supremum is taken over all shift-invariant probability measures $\mu$ on~$\cA^\nn$.
\end{lemma}

\begin{proof}
    If $\alpha \leq 0$, then $p \in \rr_+ \mapsto p^{-\alpha}$ is nondecreasing, so~UD implies
    \begin{align*}
        q_{n+\tau+m}(\alpha)
            &= \ln \sum_{a \in \supp \P_{n+\tau+m}} \P[a]^{-\alpha} \Q[a] \\
            &\leq \ln \sum_{a \in \supp \P_{n}} \sum_{\xi \in \supp \P_{\tau}} \sum_{b\in\supp \P_{m}} (\Exp{k_n}\P[a]\P[b])^{-\alpha} \Exp{k_n}\Q[a] \Q[b]\\
            &= q_n(\alpha) + q_m(\alpha) + (1-\alpha) k_n + \ln (\#\supp \P_{\tau}).
    \end{align*}
     Because $k_n = o(n)$ and $\ln|\supp \P_{\tau}| \leq \tau \ln (\#\cA)$, 
     the limit exists in $[-\infty,0]$ and equals
     \[ 
        q(\alpha) = \inf_{n\in\nn} \frac{q_n(\alpha) + (1-\alpha)k_n + \tau \ln(\#\cA)}{n + \tau}
     \]
     by a gapped version of Fekete's lemma; see~\cite[\S{II}]{Ra23}. 

    We now turn to the variational representation. Note that the variational principle for measures on finite sets\footnote{ In the thermodynamic formalism, the ``potential'' is the function on~$\cA^n$ defined by $a \mapsto \ln \P[a]$, the ``reference measure'' on~$\cA^n$ is the $n$-th marginal of~$\Q$, and the ``inverse temperature'' is~$\alpha$.} yields
    \begin{align}
    \label{eq:finite-n-vp}
        q_n(\alpha) = \sup_{\mu_n}\left(-\alpha \int \ln \P_n \dd\mu_n - \int [-\ln \Q_n + \ln \mu_n] \dd\mu_n  \right),
    \end{align}
    {where the supremum is taken over all probability measure $\mu_n$ on $\cA^n$.} Dividing by~$n$ and taking $n\to\infty$ for (the marginals of) every invariant measure~$\mu$ and then taking a supremum, we find the following variational inequality:
    \begin{align}
    \label{eq:vp-ineq}
        q(\alpha) \geq \sup_{\mu \text{ invar}} \left(\int \alpha h_\P - [h_\Q - h_\mu] \dd\mu  \right).
    \end{align}
    We have used Lemma~\ref{lem:cross-SMB-gap} and the Shannon--McMillan--Breiman theorem.
    The fact that this is actually an equality can be proved by showing that maximizers for~\eqref{eq:finite-n-vp}, whose form is explicit, admit\,---\,once properly extended and periodized\,---\,a subsequence converging weakly as $n\to\infty$ to a maximizer for~\eqref{eq:vp-ineq}, in a way that parallels~\cite[\S{4.1}]{BJPP18}; also see e.g.~\cite[\S{III.4}]{Sim}. At the technical level, the proof requires an adaptation of Lemma~2.3 in~\cite[\S{2}]{CFH08} on subadditive sequences of functions to $(\ln \P_n)_{n=1}^\infty$, which instead satisfies the gapped subadditivity condition considered in~\cite{Ra23}.
\end{proof}

\subsection{Ingredients I and II}

Within our framework, the situation for Ingredients I and II is rather satisfying, as exhibited by the following theorem.

\begin{theorem}
\label{thm:abstract-modified-parsing}
    Suppose that $\P$ satisfies~\ref{it:SLD}, % and~\ref{it:SE}, 
    that $\Q$ satisfies~\ref{it:UD}, and that~\ref{it:ND} holds. Then, for all $\epsilon \in (0,1)$ and $\lambda > 0$, we have the following properties for $N$ large enough:
    \begin{enumerate}
        \item[i.] with probability at least $1-3N^{-2}$,
        $$(1-\epsilon)[c_N-1] \ln N 
                \leq -\sum_{i=1}^{c_N-1} \ln \P[\overline{w}^{(i,N)}];
        $$
        
        \item[ii.] with probability at least $1-3N^{-2}$, 
        $$
            -\sum_{i=1}^{c_N} \ln \P[\underline{w}^{(i,N)}]
            \leq (1+\epsilon)c_N \ln N.
        $$
    \end{enumerate}
\end{theorem}

\begin{proof}
    Let $\epsilon \in (0,1)$ be arbitrary. 
    \begin{enumerate}
        \item[i.] We begin with a simple observation: if $[c_N(y|x) - 1] \ln N > -\frac{1}{1-\epsilon} \sum_{i=1}^{c_N - 1} \ln \P[\overline{w}^{(i)}]$ and the bounds in~\eqref{eq:a-priori-bounds-a} hold, then there must exist $j \in\{0,1,\dotsc,N-1\}$ and $\ell \in\{1,2,\dotsc,\ell_+\}$ such that $W_{\ell}(T^jy, x) \geq N - \ell$ and $\P[y_{j+1}^{j+\ell}] > N^{-1+\epsilon}$. Importantly, the implication about the waiting time comes from the definition of the mZM parsing. Indeed, if a subword $\overline w^{(i)}$ appears as a parsed word in $y_1^N$, then by definition it cannot appear as a subword inside $x_1^N$.
        Therefore, by Lemma~\ref{lem:a-priori-bounds-a} and a union bound, we have
        \begin{align*}
            &(\P\otimes\Q) \left\{(x,y) : [c_N(y|x) - 1] \ln N > -\frac{1}{1-\epsilon} \sum_{i=1}^{c_N - 1} \ln \P[\overline{w}^{(i)}] \right\} 
            \\ &\leq 
            2N^{-2} + \sum_{j=0}^{N-1} \sum_{\ell = 1}^{\ell_+} (\P\otimes\Q) \left\{(x,y) : W_{\ell}(T^jy, x) \geq N - \ell_+ \text{ and } \P[y_{j+1}^{j+\ell}] > N^{-1+\epsilon}\right\}.
        \end{align*}
        Now, recalling the Kontoyiannis argument for measures satisfying~\ref{it:SLD}~\cite[\S{3.1}]{CR23} and the fact that $\tau_\ell = o(\ell)$, we have
        \begin{align*}
            &(\P\otimes\Q) \left\{(x,y) : W_{\ell}(y, x) \geq N - \ell_+ \text{ and } \P[y_{1}^{\ell}] > N^{-1+\epsilon}\right\} 
            \\
            &\hspace{.5\textwidth} 
            \leq (1-\Exp{-k_\ell}N^{-1+\epsilon})^{\left\lfloor\frac{N-\ell_+-1}{\ell + \tau_{\ell}}\right\rfloor} 
            \\
            &\hspace{.5\textwidth} 
            \leq \exp\left(-\frac{N^{\epsilon}}{3\ell\Exp{k_{\ell}}}\right).
        \end{align*} 
        Hence, recalling that $k_\ell = o(\ell)$  and that we are only considering $\ell \leq \ell_+ = O(\ln N)$, by shift invariance,
        \begin{align*}
            &(\P\otimes\Q) \left\{(x,y) : [c_N(y|x) - 1] \ln N > -\frac{1}{1-\epsilon} \sum_{i=1}^{c_N - 1} \ln \P[\overline{w}^{(i)}] \right\}  
                \\
                &\hspace{.5\textwidth} 
                \leq 2N^{-2} + \Exp{-N^{\frac{\epsilon}{2}} } 
                \\
                &\hspace{.5\textwidth} 
                \leq 3N^{-2},
        \end{align*}
        for $N$ large enough, as desired.

    \item[ii.]    
       Let $r \in (0,1)$ and $\lambda \geq \tau$ be arbitrary. Let $c^+$ be a multiple of $2 \lceil N^r \rceil$ and split $\{1,2,\dotsc,c^+\}$ into many pairs of interlaced families of $\lceil N^r \rceil$ indices: $(\mathcal{I}_s)_{s=1}^{s_{c^+}}$ are defined by
        \begin{align*}
        \mathcal{I}_1 = \{1, 3, 5, \dotsc, 1+ 2(\lceil N^r\rceil-1) \}
        \qquad\text{and}\qquad
        \mathcal{I}_2 = \{2, 4, 6, \dotsc, 2+ 2(\lceil N^r\rceil-1) \}, 
        \end{align*}
        and then
        \begin{align*}
            \mathcal{I}_{2k+1} = \mathcal{I}_{1} + 2k\lceil N^r\rceil 
            \qquad\text{and}\qquad
            \mathcal{I}_{2k+2} = \mathcal{I}_{2} + 2k\lceil N^r\rceil,
        \end{align*}    
        for as long as $2(k+1)\lceil N^r \rceil\leq c^+$.
            Consider natural numbers $(L_i)_{i=1}^{c^+}$  satisfying $$\lambda \leq L_{i} - L_{i-1} \leq \ell_{+,N}$$ for each $i=1,2,\dotsc,c^+$.\footnote{Here and below, $L_0 \equiv 0$ by convention.} In particular,
            $ L_{i'-1} - L_{i} \geq \tau $
            whenever $i < i'$ belong to $\mathcal{I}_s$ for the same~$s$.
            Since 
            \begin{align*}
            (\P\otimes\Q)\left\{-\ln \P[y_{L+1}^{L'}] \geq t \text{ and } W_{L'-L}(T^{L}y,x) < N  \right\}
                &\leq \sum_{\substack{a \in \supp \P_{L'-L} \\ \P[a] \leq \Exp{-t}}} \Q[a] \cdot N \P[a] \\
                &\leq N \Exp{-t}
            \end{align*}
            and $\Q$ satisfies~\ref{it:UD}, we can appeal to the Chebyshev-like bound detailed in Lemma~\ref{lem:Cheby-like} with exponential moments of order $\varkappa = (1+\frac \epsilon 2 )^{-\frac 12}$ to deduce that, for each~$s$, we have
            \begin{equation}
            \label{eq:bound-Ch-waiting}
            \begin{split}
            &(\P\otimes\Q) \left\{ \sum_{i \in \mathcal{I}_s} -\ln \P[y_{L_{i-1}+1}^{L_{i}}] \geq  t
             \text{ and } W_{L_i - L_{i-1}}(T^{L_{i-1}}y,x) < N \text{ for each } i \in \mathcal{I}_s \right\}
             \\
             &\qquad \leq
                \left( \prod_{i\in \mathcal{I}_s}\Exp{k'_{L_{i}-L_{i-1}}} \right) \Exp{-(1+\frac \epsilon 2)^{-\frac 12} t} \left(\frac{N}{1-(1+\frac\epsilon 2)^{-\frac 12 }}\right)^{\lceil N^r \rceil} 
            \\
            &\qquad \leq 
                \exp\left(\lceil N^r \rceil \ell_{+,N} \left(\sup_{l \geq \lambda} \frac{k'_l}{l}\right) -\left( 1+\frac\epsilon 2 \right)^{-\frac 12} t +  \lceil N^r \rceil \ln N +    C_\epsilon N^r  \right),
            \end{split}
            \end{equation}
            where $k'_l \coloneqq k_l + \tau \ln |\cA|$.
            We have used the fact that the different $L_i-L_{i-1}$ for $i \in \mathcal{I}_s$ sum to at most $\lceil N^r \rceil \ell_{+,N}$.
            Below, we will use the above estimate with $t = t_N \coloneqq (1+\frac\epsilon 2) \lceil N^r\rceil \ln N$ and for $\lambda$ and $N$ large enough. 

            Now, set, as a measurable function of~$x$ and~$y$,
            \[ 
                c_N^+ \coloneqq \min\{c^+ > c_N : c^+ \text{ is a multiple of } 2 \lceil N^r \rceil\}.
            \]
            and $(\underline{w}_+^{(i,N)})_{i=1}^{c_N^+}$ the words in the mZM parsing of~$y$ with respect to $x_1^N$ that is not stopped until $c_N^+$ steps.
            Note that, under the bounds~\eqref{eq:a-priori-bounds-a} from Lemma~\ref{lem:a-priori-bounds-a}, we have
            $c_N^+ \leq c_N + 4 N^r = c_N + o(c_N)$, so
            \begin{align*}
                &(\P\otimes\Q)
                \left\{-\sum_{i=1}^{c_N} \ln \P[\underline{w}^{(i,N)}]
                > (1+\epsilon)c_N \ln N\right\}
                \\
                &\qquad 
                \leq 2N^{-2} + (\P\otimes\Q)
                \left\{-\sum_{i=1}^{c_N^+} \ln \P[\underline{w}_+^{(i,N)}]
                > \left( 1+\frac\epsilon 2 \right)c_N^+ \ln N \text{ and~\eqref{eq:a-priori-bounds-a} holds }\right\} 
                \\
                &\qquad 
                \leq 2N^{-2} + (\P\otimes\Q)\Bigg(\bigcup_{c^+}\bigcup_{s}
                \Bigg\{-\sum_{i \in \mathcal{I}_s} \ln \P[\underline{w}_+^{(i,N)}]
                > \left( 1+\frac\epsilon 2 \right) \lceil N^r \rceil\ln N, 
                    \\ & \hspace{.65\textwidth}
                    \text{\eqref{eq:a-priori-bounds-a} holds and }c_N^+ = c^+ \Bigg\}\Bigg)
            \end{align*}
            for $N$ large enough. But since the words $\underline{w}_+^{(i,N)}$ do appear in~$x_1^N$, we may use the bound~\eqref{eq:bound-Ch-waiting} to deduce that 
            \begin{align*}
            &(\P\otimes\Q)\left(\bigcup_{c^+}\bigcup_{s}
                \left\{-\sum_{i \in \mathcal{I}_s} \ln \P[\underline{w}_+^{(i,N)}]
                > \left( 1+\frac\epsilon 2 \right) \lceil N^r\rceil \ln N \text{ and~\eqref{eq:a-priori-bounds-a} holds }\right\}\right)
                \\
                &
                \leq 
                \operatorname{Comb}(N;\lambda) \exp\left(-\left(\left( 1+\frac\epsilon 2 \right)^{\frac 12} - 1\right) \lceil N^r \rceil\ln N +  \lceil N^r \rceil\ell_{+,N} \left(\sup_{l \geq \lambda} \frac{k'_l}{l}\right) + C_\epsilon N^r\right)
            \end{align*}
            for all~$N$ large enough, where $\operatorname{Comb}(N;\lambda)$ is a combinatorial factor that takes into account all possibilities for $c^+$ and $s$, and then choices of $L_{i-1}$  and $L_{i}-L_{i-1}$ for $i \in \mathcal{I}_s$. Since crude estimates give
            \begin{align*}
                \operatorname{Comb}(N;\lambda) 
                    &\leq (\lambda^{-1}N + 4 N^r) \times \frac{\lambda^{-1}N + 4 N^r }{N^r } \times N \ell_{+,N}^{\lceil N^r \rceil-1} \times \ell_{+,N}^{\lceil N^r \rceil} \\
                    &\leq \exp\left( N^r O(\ln \ln N)\right),
            \end{align*}
            the result follows. \qedhere
    \end{enumerate}
\end{proof}
\subsection{Ingredients III and IV}
\label{ssec:iii-iv}

The following propositions establish the missing ingredients for the proof of Theorem~\ref{thm:sc-gen}.

\begin{proposition}
\label{prop:diff-minus}
    Suppose that $\P$ satisfies~\ref{it:SLD}, that $\Q$ satisfies~\ref{it:UD}, %with $\tau_n=O(1)$, 
    and that~\ref{it:ND} holds. Then, almost surely,
    \[
    \liminf_{N\to\infty} \frac{1}{N-c_N}\sum_{i =1}^{c_N} -\ln \P[\underline{w}^{(i,N)}] \geq  D_-\bar{q}(0).
    \]
\end{proposition}

\begin{proof}
    We start with the case $D_-\bar{q}(0) < \infty$ and will conclude with a description of how to adapt the proof to the infinite case. Let Let $\lambda > \tau$ and $\epsilon > 0$ be arbitrary, and then fix a natural number $c \leq \lambda^{-1}N$ as well as $c-1$ natural numbers, $(L_i)_{i=1}^{c-1}$, with $\lambda \leq L_{i} - L_{i-1}$ for each $i=1,2,\dotsc,c-1$ and $L_{i}-L_{i-1} \leq \ell_{+}$ for $i=1,2,\dotsc,c$.\footnote{Here and below, $L_0 \equiv 0$ and $L_c \equiv N$ by convention.} Note that if 
    \begin{equation*}
        \sum_{i=1}^{c} -\ln \P[y_{L_{i-1}+1}^{L_{i}-1}] \leq (N-c) (D_-\bar{q}(0) - \epsilon) ,
    \end{equation*}
    then
    \begin{equation*}
        \sum_{i=1}^{c-1} -\ln \P[y_{L_{i-1}+1}^{L_{i}-\max\{1,\tau\}}] \leq (N-c) (D_-\bar{q}(0) - \epsilon)
    \end{equation*}
    as well. {We will omit keeping track of the maximum with 1 for readability.}
    Then, we can use the assumption that~$\Q$ satisfies UD to invoke the Chebyshev-like bound in Lemma~\ref{lem:Cheby-like} and deduce that
    \begin{align}
        &\Q\left\{ y : \sum_{i=1}^{c-1} -\ln \P[y_{L_{i-1}+1}^{L_{i}-\tau}] \leq (N-c) (D_-\bar{q}(0) - \epsilon) \right\} 
            \notag 
            \\ & \qquad
            \leq \exp\left(-\alpha (N-c) (D_-\bar{q}(0) - \epsilon) \right) \prod_{i=1}^{c-1} \Exp{k'_{L_i-L_{i-1}-\tau}+q_{L_i-L_{i-1}-\tau}(\alpha)}
            \label{eq:Ch-D_-}
    \end{align}
    for all $\alpha < 0$ with small enough absolute value, where $k'_l \coloneqq  k_l + \tau \ln |\cA|$. Note that, provided that $\lambda$ is large enough, we will have 
    \[ 
        \frac{N-c}{N-c\tau-\ell_+} \left(D_-\bar{q}(0) - \frac{\epsilon}{2}\right) < D_-\bar{q}(0)
    \]
    for all~$N$ large enough. Then, by the definition and convexity of~$\bar{q}$ and the fact that $k'_l = o(l)$, we may then take $\alpha < 0$ with $|\alpha|$ so small
    \[ 
         \sup_{l \geq \lambda-\tau}  \left(\frac{k'_{l}}{l} + \frac{q_{l}(\alpha)}{l} \right) < \frac{N-c}{N-c\tau-\ell_+} \alpha (D_-\bar{q}(0) - \epsilon)  
    \]
    for all $\lambda$ and $N$ large enough%
        \footnote{First, $\lambda$ is chosen as a function of~$\epsilon$, $\tau$ and $\bar{q}$ only. Then, 
        $\alpha$ is picked as a function of~$\epsilon$ and $\bar{q}$, but independently of~$N$ and $\lambda$. Therefore, we can indeed make those large enough, as a function of $\epsilon$, $\tau$, $\bar{q}$, $\lambda$, $\alpha$, $(q_\ell(\alpha))_{\ell=1}^\infty$, and $(k_\ell)_{\ell=1}^\infty$ only in this last step.}%
    ; see Figure~\ref{fig:subdiff}.
        
 \begin{figure}
        \centering
        \includegraphics{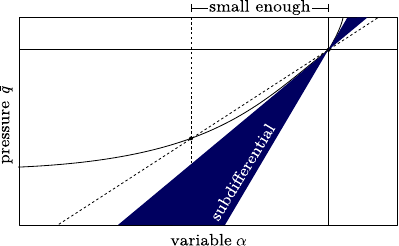}
        \caption{The dashed line has a slope that is slightly less steep than the least steep slope $D_-\bar{q}(0)$ in the subdifferential of the convex function~$\bar{q}$ at the origin, so we can find a small region to the left of the origin where the graph of~$\bar{q}$ lies below it.}
        \label{fig:subdiff}
    \end{figure}

    Since $\sum_{i=1}^{c-1} L_i - L_{i-1} - \tau \geq {N-c\tau-\ell_+}$, this implies that there exists $\delta' > 0$ such that, under the same constraints, 
    \begin{align*}
        \sum_{i=1}^{c-1} k'_{ L_i - L_{i-1} - \tau} +q_{ L_i - L_{i-1} - \tau}(\alpha) 
            &< {N-c\tau-\ell_+}
            \left( \frac{N-c}{N-c\tau-\ell_+} \alpha (D_-\bar{q}(0) - \epsilon)  - \delta'\right) \\
            &= \alpha (N-c) (D_-\bar{q}(0) - \epsilon) - \delta' (N-c\tau-\ell_+).
    \end{align*}
   Therefore, the bound~\eqref{eq:Ch-D_-} becomes
    \begin{align*}
        \Q\left\{ y : \sum_{i=1}^{c-1} -\ln \P[y_{L_{i-1}+1}^{L_{i}-\tau}] \leq (N-c) (D_-\bar{q}(0) - \epsilon) \right\} 
            &\leq \Exp{-\delta N}
    \end{align*}
    for some $\delta > 0$ and all $N$ large enough.
    All in all, using again Lemma~\ref{lem:a-priori-bounds-a} and a union bound, we have 
    \begin{align*}
        &(\P\otimes\Q)\left\{(x,y) : \sum_{i =1}^{c_N} -\ln \P[\underline{w}^{(i,N)}] \leq (N-c) (D_-\bar{q}(0) - \epsilon)\right\} \\
        &\qquad \leq 2N^{-2} + \sum_{c=1}^{\lambda^{-1}N} \sum_{\substack{(L_i)_{i} \in \{1,\dotsc,N-1\}^{c-1} \\ L_{i}-L_{i-1} \geq \lambda}}\Exp{-\delta N} \\
        &\qquad \leq 2N^{-2} + \lambda^{-1}N \times \binom{N}{\lambda^{-1} N} \times \Exp{-\delta N} \\
        &\qquad \leq 2N^{-2} + \exp\left(\ln(\lambda^{-1} N) + \frac{1 + \ln \lambda}{\lambda}N - \delta N \right).
    \end{align*}
    Since the right-hand side is summable in~$N$, the result in the case $D_-\bar{q}(0)<\infty$ follows from the Borel--Cantelli lemma and then taking $\epsilon \to 0$ along some discrete sequence.
    Finally, in the case $D_-\bar{q}(0) = \infty$, we simply replace ``$D_-\bar{q}(0) - \epsilon$'' with ``$\epsilon^{-1}$'' and again take $\epsilon \to 0$ along some discrete sequence.
\end{proof}

\begin{remark}\label{rmk:corr}
    As $c_N = o(N)$ by Remark~\ref{rem:cN-is-oN}, we could have proven the result with $N$ instead of $N- c_N$ in the denominator. However, since $\ell_{+,N} \ll c_N \ll N$, we see that, when $\tau = 0$ or $\tau = 1$, the factor carried throughout the proof to pass from $N-c \max\{\tau,1\}-\ell_{+,N}$ to the desired $N-c$ is significantly closer to~$1$ than the factor that would be needed to pass to~$N$. Hence, the denominator $N-c_N$ is in some sense more natural for this proof. This is consistent with our numerical investigation of the convergence in Section~\ref{ssec:hmm}.
\end{remark}

\begin{proposition}
\label{prop:diff-plus}
    Suppose that $\P$ satisfies~\ref{it:SLD}, that $\Q$ satisfies~\ref{it:UD},
    and that~\ref{it:ND} holds. Then, almost surely,
    \[
    \limsup_{N\to\infty} \frac{1}{N}\sum_{i =1}^{c_N} -\ln \P[\overline{w}^{(i,N)}] \leq  D_+\bar{q}(0).
    \]
\end{proposition}
\begin{proof}
    Note that if $D_+\bar{q}(0) = \infty$, then the proposed inequality is trivially true. Hence, we will assume that $D_+\bar{q}(0) < \infty$ and let $\epsilon > 0$ be arbitrary.
    Now, fix a natural number $c \leq \lambda^{-1}N$ as well as $c$ natural numbers, $(L_i)_{i=1}^{c}$, with $\lambda \leq L_{i} - L_{i-1} \leq \ell_{+}$ and $N \leq L_c \leq N + \ell_{+}$.%
        \footnote{Again $L_0 \equiv 0$ by convention.}  
    Now, let $\mathcal{I}_1$ [resp.~$\mathcal{I}_2$] be the odd [resp.\ even] indices $i \leq c$.
    Note that if 
    \begin{equation}
    \label{eq:contr-ub-Dplus}
        \sum_{i =1}^c -\ln \P[y_{L_{i-1}+1}^{L_{i}}] \geq N (D_+\bar{q}(0) + \epsilon),
    \end{equation}
    and $\nu \in (0,1)$, then 
    \begin{equation}
    \label{eq:contr-ub-Dplus-s}
        \sum_{i \in \mathcal{I}_s} -\ln \P[y_{L_{i-1}+1}^{L_{i}}] \geq (1-\nu) \sum_{i \in \mathcal{I}_s} (L_i-L_{i-1}) (D_+\bar{q}(0) + \epsilon) + \nu \frac{N}{2} (D_+\bar{q}(0) + \epsilon).
    \end{equation}
    for either $s=1$ or $s=2$ (possibly both).
    Provided that $\lambda \geq \max\{2,\tau\}$, we can use UD and the Chebyshev-like bound in Lemma~\ref{lem:Cheby-like} to deduce that 
    \begin{align*}
        &\Q\left\{y : \sum_{i \in \mathcal{I}_s} -\ln \P[y_{L_{i-1}+1}^{L_{i}}] \geq (1-\nu) \sum_{i \in \mathcal{I}_s} (L_i-L_{i-1}) (D_+\bar{q}(0) + \epsilon) + \nu \frac{N}{2} (D_+\bar{q}(0) + \epsilon) \right\}
        \\
        &\qquad \leq \exp\left(-\alpha(1-\nu) \sum_{i \in \mathcal{I}_s} (L_i-L_{i-1}) (D_+\bar{q}(0) + \epsilon) - \alpha\nu \frac{N}{2} (D_+\bar{q}(0) + \epsilon)\right)
            \\ & \hspace{.5\textwidth} \prod_{i \in \mathcal{I}_s} \Exp{k'_{L_i-L_{i-1}}(\alpha)+q_{L_i-L_{i-1}}(\alpha)}
    \end{align*}
    for all $\alpha > 0$ small enough. 
    From now on, we will require
    \begin{equation}
    \label{eq:nu-cond-1}
     \nu < 1-\frac{D_+\bar{q}(0) + \tfrac 12 \epsilon}{D_+\bar{q}(0) + \epsilon}
    \end{equation}
    so that 
    \begin{align*}
        &\Q\left\{y : \sum_{i \in \mathcal{I}_s} -\ln \P[y_{L_{i-1}+1}^{L_{i}}] \geq (1-\nu) \sum_{i \in \mathcal{I}_s} (L_i-L_{i-1}) (D_+\bar{q}(0) + \epsilon) + \nu \frac{N}{2} (D_+\bar{q}(0) + \epsilon) \right\}
        \\
        &\qquad \leq \exp\left(-\alpha \sum_{i \in \mathcal{I}_s} (L_i-L_{i-1}) (D_+\bar{q}(0) + \tfrac 12 \epsilon) - \alpha\nu \frac{N}{2} (D_+\bar{q}(0) + \epsilon)\right) 
            \\ & \hspace{.5\textwidth} 
            \prod_{i \in \mathcal{I}_s^*} \Exp{k'_{L_i-L_{i-1}}(\alpha)} \prod_{i \in \mathcal{I}_s} \Exp{q_{L_i-L_{i-1}}(\alpha)}
    \end{align*}
    By the definition and convexity of~$\bar{q}$ and the fact that $k'_l = o(l)$, we may then take $\alpha > 0$ so small that 
    \[ 
        \alpha \left( D_+\bar{q}(0) + \frac \epsilon 2 \right) \geq \sup_{l \geq \lambda} \left(\frac{q_{l}(\alpha)}{l} + \frac{k'_{l}}{l}\right).
    \]
    for all $\lambda$ large enough\footnote{The choice of $\alpha$ is first made as function of~$\epsilon$ and $\bar{q}$, but independently of~$\lambda$, and we then make the latter large enough as a function of $\epsilon$, $\bar{q}$, $\alpha$, $(q_\ell(\alpha))_{\ell=1}^\infty$ and $(k_\ell)_{\ell=1}^\infty$ only.}.
    Therefore,
    \begin{align*}
        &\Q\left\{y : \sum_{i \in \mathcal{I}_s} -\ln \P[y_{L_{i-1}+1}^{L_{i}}] \geq (1-\nu) \sum_{i \in \mathcal{I}_s} (L_i-L_{i-1}) (D_+\bar{q}(0) + \epsilon) + \nu \frac{N}{2} (D_+\bar{q}(0) + \epsilon) \right\}
        \\
            &\hspace{.5\textwidth} 
            \leq \exp\left(-\frac{\alpha\nu N}{2} (D_+\bar{q}(0) + \epsilon)\right). 
    \end{align*}
    Denote by $(\overline w^{(i,N)}_+)_{i=1}^{c_N}$ the words in a mZM parsing of $y$ with respect to $x_1^N$ that is not stopped until the bona fide end of the step~$c_N$. In other words, only the last word may change and it is now allowed to extend past $y_1^N$ in order to avoid having match in~$x_1^N$.
    All in all, using again Lemma~\ref{lem:a-priori-bounds-a} and a union bound, we have
    \begin{align*}
        &(\P\otimes\Q)\left\{\sum_{i =1}^{c_N} -\ln \P[\overline{w}^{(i,N)}] \geq N (D_+\bar{q}(0) + \epsilon)\right\} \\
        &\qquad
        \leq 2N^{-2} + (\P\otimes\Q)\Bigg( \bigcup_{c=1}^{\lambda^{-1}N} \Bigg\{ \sum_{i =1}^{c_N} -\ln \P[\overline{w}_+^{(i,N)}] \geq N (D_+\bar{q}(0) + \epsilon) 
            \\ &\hspace{.5\textwidth}
            \text{ and~\eqref{eq:a-priori-bounds-a} holds and } c_N = c \Bigg\} \Bigg)\\
        &\qquad \leq 2N^{-2} + \sum_{c=1}^{\lambda^{-1}N} \sum_{s=1}^2 \sum_{\substack{(L_i)_{i} \in \{1,\dotsc,N-1\}^{c} \\ L_{i}-L_{i-1} \geq \lambda}}\exp\left(-\frac{\alpha\nu N}{2} (D_+\bar{q}(0) + \epsilon)\right) \\
        &\qquad \leq 2N^{-2} + \exp\left(\ln(\lambda^{-1} N) + \ln 2 + \frac{1 + \ln \lambda}{\lambda}N -\frac{\alpha\nu N}{2} (D_+\bar{q}(0) + \epsilon)\right).
    \end{align*}
    This will be summable provided that 
    \begin{equation}
    \nu > \frac{1+\ln \lambda}{\lambda} \frac{2}{\alpha (D_+\bar{q}(0)+\epsilon)},
    \end{equation}
    which is consistent with our earlier requirement~\eqref{eq:nu-cond-1} because $\lambda$ can be taken to be as large as desired without changing the values of $\epsilon$ and $\alpha$. Hence, the result follows from the Borel--Cantelli lemma and then taking $\epsilon \to 0$ along some discrete sequence.
\end{proof}

As mentioned at the beginning of the section, Theorem~\ref{thm:abstract-modified-parsing}, Proposition~\ref{prop:diff-minus} and Proposition~\ref{prop:diff-plus} imply Theorem~\ref{thm:sc-gen}, and Corollary~\ref{cor:when-diff} follows.

\section{A Property of the left derivative}
\label{ssec:left-deriv}

 Before we change point of view on the problem in Section~\ref{sec:open-prob}, we present an important technical ingredient in the form of a theorem which will be useful in the important family of hidden-Markov measures and quantum measurements in Section~\ref{ssec:hmm}.
\begin{theorem}
\label{thm:left-der}
    Suppose that~$\P$ satisfies~\ref{it:UD}. If, in addition, $\Q$ is ergodic and satisfies~\ref{it:UD} with 
    $k_n = O(1)$, then 
    \[ 
        D_- q(0) = \Sc(\Q|\P).
    \]
\end{theorem}

\begin{proof}[Proof.]
    Recall that the limit $q$ exists for all~$\alpha \leq 0$ and has the variational expression~\eqref{eq:alpha-var-p} by Lemma~\ref{lem:q-exists}.  
    By gapped versions of Fekete's lemma and Kingman's theorem, one can write the maps $\mu \mapsto \alpha \int h_\P \dd\mu$ and $\mu \mapsto -\int h_\Q \dd\mu$ as infima of families of continuous maps on the space of shift-invariant Borel measures on~$\cA^\nn$ equipped with the topology of weak convergence. Hence, these two maps are upper semicontinuous. It is also possible to show that they are affine. Both these properties are well known for the Kolmogorov--Sinai entropy map $\mu \mapsto \int h_\mu \dd\mu$ on shift spaces. In particular, the set of maximizers is nonempty and convex.
    
    {We follow the same strategy as for the analogous result in~\cite[\S{4.1}]{BJPP18}.} Using elementary properties of convex functions and this variational expression, one can show that 
    \begin{align*}
        D_-q(0) &= \sup_{\alpha < 0} \frac{q(\alpha)}{\alpha} \\
        &\leq \int h_\P \dd\Q.
    \end{align*} 
    Note that the right-hand side is the desired specific cross entropy by Lemma~\ref{lem:cross-SMB-gap}. {If $D_-q(0)=\infty$, then the theorem is already established. So, we will assume otherwise for the remainder of the proof.}
    
    Consider now a sequence of maximizers $\mu(\alpha_k)$ for~\eqref{eq:alpha-var-p} with $\alpha = \alpha_k$ such that $\alpha_k$ increases to 0 as $k\to\infty$, avoiding the points of discontinuity of~$q'$. Up to extracting a subsequence, we may assume that $\mu(\alpha_k)$ converges weakly to some shift-invariant~$\mu(0)$. Using upper semicontinuity of the aforementioned maps and the continuity of~$q$ on~$(-\infty,0]$, one can show that this $\mu(0)$ is a maximizer for~\eqref{eq:alpha-var-p} with $\alpha = 0$. Again using upper semicontinuity one can then show that
    \begin{align*}
        D_-q(0) &\geq \lim_{k\to\infty} \int h_\P \dd\mu(\alpha_k) \\
            &\geq \inf_{\mu \text{ 0-eq}} \int h_\P \dd\mu,
    \end{align*}
    where the infimum is taken over all~$\mu$ in the set of maximizers for~\eqref{eq:alpha-var-p} with $\alpha = 0$.\footnote{In the thermodynamic formalism, we call these maximizers ``$\alpha$-equilibrium measures'', hence the ``0-eq'' in the notation.}
    Hence, the theorem will be proved if we can show that $\Q$ is the \emph{unique} maximizer for~\eqref{eq:alpha-var-p} with $\alpha = 0$. {This can again be done following the proof of the analogous result in~\cite[\S{4.1}]{BJPP18}, which is in turn along the lines of classical arguments of Ruelle summarized e.g.\ in~\cite[\S{III.8}]{Sim}. We will still provide the outline of the proof for completeness.}
    
    Let $\mu$ be any such maximizer, i.e.\ a shift-invariant measure such that
    \begin{equation}
    \label{eq:0-eq-q}
        0 = q(0) = -\int [h_\Q - h_\mu] \dd\mu.
    \end{equation}
    By~\ref{it:UD} for $\Q$, 
    \[
        F_n \coloneqq  \int [\ln \Q_n - \ln \mu_n] \dd\mu 
    \]
    defines a gapped, subadditive sequence\footnote{ Note that $F_n \in [-\infty,0]$ by an elementary convexity argument and that, a priori, if $F_n = -\infty$ for some $n$ (due to $\supp \mu_n \not\subseteq \supp \Q_n$), then $F_{n+m} = -\infty$ for all $m$. We will soon see that, a posteriori, this is ruled out.} {in the sense that 
    \begin{align*}
        F_{n+\tau+m} &= \int [\ln \Q_{n+\tau+m} - \ln \mu_{n+\tau+m}] \dd\mu \\
        &= \sum_{a \in \cA^n}\sum_{\xi\in\cA^{\tau}} \sum_{b \in \cA^m} \ln \Q[a\xi b] \ \mu[a\xi b] + \sum_{a \in \cA^n}\sum_{\xi\in\cA^{\tau}} \sum_{b \in \cA^m} - \ln \mu[a\xi b] \ \mu[a\xi b] \\
        &\leq k_n + \sum_{a \in \cA^n} \ln \Q[a]\ \mu[a] +  \sum_{b \in \cA^m} \ln\Q[b] \ \mu[b] \\
        &\qquad\qquad {} +  \sum_{a \in \cA^n} - \ln \mu[a] \ \mu[a] + \sum_{\xi\in\cA^{\tau}} - \ln \mu[\xi] \ \mu[\xi] + \sum_{b \in \cA^m} - \ln \mu[b] \ \mu[b] \\
        &\leq k_n + F_n + F_m + \tau \ln (\# A).
    \end{align*} 
    For the triple sum involving $\mu$ only, we have used subadditivity of entropy for measures on products of finite sets.} By combining gapped versions of Fekete's lemma and Kingman's theorem~\cite[\S\S{II--III}]{Ra23} with the Shannon--McMillan--Breiman theorem, we then have
    \begin{align}
    \label{eq:cons-Gingman}
        \int [h_\Q - h_\mu] \dd\mu = \inf_{n\in\nn} \frac{F_n + k_n + \tau \ln (\# A)}{n + \tau}.
    \end{align}
    Combining~\eqref{eq:0-eq-q} and~\eqref{eq:cons-Gingman}, we deduce that
    \[
       0 \leq -F_n  \leq k_n + \tau \ln (\# A).
    \] 
    for all~$n\in\nn$. Reinterpreting $-F_n$ as the\,---\,\emph{not} specific\,---\,relative entropy of the $n$-th marginal of~$\mu$ with respect to the $n$-th marginal of~$\Q$, we deduce that $k_n$
    being $O(1)$ implies $\mu \ll \Q$. But if $\Q$ is ergodic, then this in turn implies that $\mu = \Q$.
\end{proof}

\begin{remark}
    The hypothesis $k_n = O(1)$\,---\,as opposed to $k_n = o(n)$ as in the definition of~\ref{it:UD}\,---\,is crucial, as demonstrated by counterexamples in Section~\ref{sec:Examples}.
\end{remark}

\section{A Rigidity problem}
\label{sec:open-prob}

Once Proposition~\ref{thm:abstract-modified-parsing} has established Ingredients~I and~II, one can change point of view and focus on the ``rigidity'' of Lemma~\ref{lem:cross-SMB-gap}. Here, ``rigidity'' of the result is to be understood in the sense of the following problem on what happens to the convergence in this lemma as we cut~$y_1^N$ into (mostly) reasonably long words. 
We work on the probability space $\cA^\nn \otimes \cA^\nn$ equipped with the usual $\sigma$-algebra and the product measure~$\P\otimes\Q$. We say that two random arrays $((\ell_{i,N})_{i=1}^{c_N})_{N=1}^\infty$ and $((\mu_{i,N})_{i=0}^{c_N})_{N=1}^\infty$ of nonnegative integers with 
\[ 
    \mu_{0,N} + \ell_{1,N} + \mu_{1,N} + \ell_{2,N} + \dotsb + \mu_{c_{N}-1,N} + \ell_{c_N,N} + \mu_{c_{N},N} = N
\]
\emph{produce reasonable cuts} if the following properties hold on a set of probability 1:
    \begin{enumerate}
        \item[i.] for every $\lambda \in \nn$, we have $\lambda < \ell_{i,N}$
        whenever $N$ is large enough and $i \in \{1,2,\dotsc,c_N-1\}$;

        \item[ii.] for every $\epsilon > 0$, we have $\mu_{0,N} + \mu_{1,N} + \dotsb + \mu_{c_N,N} < \epsilon N$ whenever $N$ is large enough.
    \end{enumerate}
    We will then use the shorthand 
    \[ 
        L_{i,N} \coloneqq  \sum_{i'=0}^{i}\mu_{i',N} + \sum_{i'=1}^{i} \ell_{i',N}.
    \]
    Whether taking $\ell_{i,N} = |\overline{w}^{(i,N)}|$ for all $i$ and $\mu_{i,N} = 0$ for all $i$, or $\ell_{i,N} = |\underline{w}^{(i,N)}|$ for all $i$ and $\mu_{0,N} = 0$, $\mu_{i,N} = 1$ for all~$i> 0$, Lemma~\ref{lem:a-priori-bounds-a} guarantees that mZM parsings produce reasonable cuts in Theorem~\ref{thm:abstract-modified-parsing}.

\begin{description}
    \item[Problem.] \em Find natural properties of the measures~$\P$ and~$\Q$ that guarantee the following property: whenever $((\ell_{i,N})_{i=1}^{c_N})_{N=1}^\infty$ and $((\mu_{i,N})_{i=0}^{c_N})_{N=1}^\infty$ produce reasonable cuts, then
    \[
        \lim_{N\to\infty} \frac{-\sum_{i=1}^{c_N} \ln\P\left[y_{L_{i-1,N}+1}^{L_{i-1,N}+\ell_{i,N}}\right]}{\sum_{i=1}^{c_N} \ell_{i,N}} = \Sc(\Q|\P)
    \]
    $\Q$-almost surely.
\end{description}

One of the difficulties faced in tackling this problem is that Lemma~\ref{lem:cross-SMB-gap} itself comes with no concrete estimate and that a naive large-deviation approach\,---\,which is implicitly behind Propositions~\ref{prop:diff-minus} and~\ref{prop:diff-plus}\,---\,fails to distinguish between $\Sc(\Q|\P)$ and other points that could be in the subdifferential of the pressure at the origin. 
In some sense, an answer to this problem would only be satisfactory if it at least allowed one to recover the special cases that we are about to discuss over the course of the lemmas below. And ideally, the proof techniques would be unified.

We first show that partial results regarding this open problem\,---\,in conjunction with Lemma~\ref{lem:a-priori-bounds-a}, Theorem~\ref{thm:abstract-modified-parsing} and Lemma~\ref{lem:cross-SMB-gap}\,---\,provide improvements over~\eqref{eq:subdiff-lb} or~\eqref{eq:subdiff-ub} without differentiability. For our first lemma, the improvement is perhaps most interesting when further combined with Theorem~\ref{thm:left-der}; see Corollary~\ref{cor:UD-tau-0} and its use in Section~\ref{ssec:hmm}.

\begin{lemma}
\label{lem:UD-for-III}
    Suppose that $((\ell_{i,N})_{i=1}^{c_N})_{N=1}^\infty$ and $((\mu_{i,N})_{i=0}^{c_N})_{N=1}^\infty$ produce reasonable cuts. If $\P$ satisfies~\ref{it:UD} with $\tau = 0$ and $\Q$ is ergodic, then
    \[
        \limsup_{N\to\infty} \frac{-\sum_{i=1}^{c_N} \ln\P\left[y_{L_{i-1,N}+1}^{L_{i-1,N}+\ell_{i,N}}\right]}{\sum_{i=1}^{c_N} \ell_{i,N}} \leq \Sc(\Q|\P)
    \]
    $\Q$-almost surely.
\end{lemma}

\begin{proof}
    Let $\epsilon>0$ be arbitrary and pick $\lambda \in$ large enough that $k_l \leq \epsilon l$ whenever $l \geq \lambda$.
    The upper-decoupling assumption and Property~i give that, almost surely,
    \begin{align*}
        \ln \P[y_1^{N}] 
            &\leq \sum_{i=1}^{c_N} k_{\mu_{i-1,N}} + \ln\P\left[y_{L_{i-1,N}+1}^{L_{i-1,N}+\ell_{i,N}}\right] + k_{\ell_{i,N}} \\
            &\leq \sum_{i=1}^{c_N} \mu_{i-1,N} \left(\sup_{l \in \nn} \frac{k_l}{l}\right) + \sum_{i=1}^{c_N} \ln\P\left[y_{L_{i-1,N}+1}^{L_{i-1,N}+\ell_{i,N}}\right] + \sum_{i=1}^{c_N} \epsilon \ell_{i,N}.
    \end{align*}
    for $N$ large enough.
    Now using Property~ii, we deduce that, almost surely,
    \begin{align*}
        \ln \P[y_1^{N}] 
            &\leq \epsilon N \left(\sup_{l \in \nn} \frac{k_l}{l}\right) + \sum_{i=1}^{c_N} \ln\P\left[y_{L_{i-1,N}+1}^{L_{i-1,N}+\ell_{i,N}}\right] + \epsilon N.
    \end{align*}
    for $N$ large enough.
    Dividing by $\sum_{i=1}^{c_N} \ell_{i,N}$ and using Property~ii again, we deduce that 
    \[
        \limsup_{N\to\infty} \frac{-\sum_{i=1}^{c_N} \ln\P\left[y_{L_{i-1,N}+1}^{L_{i-1,N}+\ell_{i,N}}\right]}{\sum_{i=1}^{c_N} \ell_{i,N}} \leq \limsup_{N\to\infty} \frac{-\ln \P[y_1^{N}]}{(1-\epsilon)N} + \epsilon N \left(\sup_{l \in \nn} \frac{k_l}{l} + 1\right)
    \]
    Taking~$\epsilon \to 0$, the proposed bound then follows from Lemma~\ref{lem:cross-SMB-gap}.
\end{proof}

\begin{remark}
    By a very similar argument, the conclusion still holds for~\ref{it:UD} with $\tau > 0$ \emph{and} the following lower-bound: 
    \begin{equation}
    \label{eq:add-letter}
        \inf\left\{ \frac{\P[ab]}{\P[a]} : a \in \supp \P_n, b \in \supp \P_1, ab \in \supp \P_{n+1} \right\} \geq \Exp{-k_n}
    \end{equation} 
    for every $n \in \nn$. {While such a lower bound is immediate if the measure~$\P$ satisfies the weak Gibbs condition given below, it could fail for some important examples in which $\tau = 0$.}
\end{remark}

\begin{corollary}
\label{cor:UD-tau-0}
    If, in addition to the hypotheses of Theorem~\ref{thm:sc-gen}, the measure $\P$ satisfies~\ref{it:UD} with $\tau = 0$ and $\Q$ is ergodic, then 
    \begin{subequations}
    \begin{alignat}{1}
        D_- \overline{q}(0)
        &\leq \liminf_{N\to\infty} Q_N(y,x)
        \label{eq:UD0-lb}
        \\
        &\leq \limsup_{N\to\infty} Q_N(y,x) \leq \Sc(\Q|\P)
        \label{eq:UD0-ub}
    \end{alignat}
    \end{subequations}
    for $(\P\otimes\Q)$-almost every~$(x,y)$.
    If, moreover, $\Q$ satisfies~\ref{it:UD} with $k_n=O(1)$, then 
    \begin{equation}
        \lim_{N\to\infty} Q_N(y,x) = \Sc(\Q|\P)
    \end{equation}
    for $(\P\otimes\Q)$-almost every~$(x,y)$.
\end{corollary}

Our second partial result concerns the \emph{weak Gibbs} (WG) class of Yuri on a subshift $\Omega'\subseteq \cA^\nn$, i.e.\ the class of measures~$\P$ for which there exists a continuous function~$\phi : \Omega' \to \rr$ and a nondecreasing $\Exp{o(n)}$-sequence $(K_n)_{n=1}^\infty$ such that 
\[ 
    K_n^{-1} \Exp{\sum_{j=0}^{n-1} \phi(T^jy) - np_{\textnormal{top}}(\phi)} \leq \P[y_1^n] \leq K_n \Exp{\sum_{j=0}^{n-1} \phi(T^jy) - np_{\textnormal{top}}(\phi)}
\]
for every $y \in \Omega'$. Here, $p_{\textnormal{top}}(\phi)$, known as the topological pressure, is completely determined by~$\phi$, and provides the appropriate normalization. On subshifts {with suitable specification properties}, $g$-measures and  equilibrium measures for absolutely summable interactions are WG; we do not dwell on this technical issue and instead refer the reader to~\cite{PS20}. While Lemma~\ref{lem:cross-SMB-gap} does not directly apply to WG measures, a well-known analogue provides the same conclusion.

The next lemma\,---\,in conjunction with Lemmas~\ref{lem:a-priori-bounds-a} and Theorem~\ref{thm:abstract-modified-parsing}\,---\,shows again that progress on this open problem can be beneficial.

\begin{lemma}
\label{lem:WG-for-IV}
    Suppose that $((\ell_{i,N})_{i=1}^{c_N})_{N=1}^\infty$ and $((\mu_{i,N})_{i=0}^{c_N})_{N=1}^\infty$ produce reasonable cuts. If $\P$ is WG and is $\Q$ ergodic, then
    \[
        \lim_{N\to\infty} \frac{-\sum_{i=1}^{c_N} \ln\P\left[y_{L_{i-1,N}+1}^{L_{i-1,N}+\ell_{i,N}}\right]}{\sum_{i=1}^{c_N} \ell_{i,N}} = \Sc(\Q|\P)
    \]
    $\Q$-almost surely.
\end{lemma}

\begin{proof}
    Without loss of generality, we assume that $p_{\textnormal{top}}(\phi) = 0$. Let $\epsilon > 0$ be arbitrary and pick $\lambda \in \nn$ large enough that $\ln K_l \leq \epsilon l$ for all~$l \geq \lambda$.
    With any $y \in [y_1^N]$, we have by WG and Property~i of the reasonable cuts that
    \begin{align*}
        \ln \P[y_1^{N}]
            &\leq \ln K_N + \sum_{j=0}^{N-1} \phi(T^jy) \\
            &= \ln K_N +\sum_{j=0}^{\mu_0-1}\phi(T^j y) + \sum_{i=1}^{c_N} \sum_{j=L_{i-1}}^{L_{i-1} + \ell_i + \mu_i - 1} \phi(T^jy) \\
            &\leq \ln K_N + \mu_0\|\phi\|_\infty + \sum_{i=1}^{c_N} \left[\left(\sum_{j=L_{i-1}}^{L_{i-1} + \ell_i - 1} \phi(T^jy)\right) + \mu_i \|\phi\|_\infty \right] \\
            &\leq \ln K_N + \sum_{i=0}^{c_N} \mu_i \|\phi\|_\infty + \sum_{i=1}^{c_N} \left(\ln K_{\ell_i} + \ln \P\left[y_{L_{i-1,N}+1}^{L_{i-1,N}+\ell_{i,N}}\right]\right) \\
            &\leq \ln K_N + \sum_{i=0}^{c_N} \mu_i \|\phi\|_\infty + \sum_{i=1}^{c_N} \epsilon \ell_i + \sum_{i=1}^{c_N} \ln \P\left[y_{L_{i-1,N}+1}^{L_{i-1,N}+\ell_{i,N}}\right] 
    \end{align*}
    for $N$ large enough.
    Now using Property~ii of the reasonable cuts yields that, almost surely,
    \begin{align*}
        \ln \P[y_1^{N}]
            &\leq \ln K_N  + \epsilon N \|\phi\|_\infty + \epsilon N + \sum_{i=1}^{c_N} \ln \P\left[y_{L_{i-1,N}+1}^{L_{i-1,N}+\ell_{i,N}}\right] 
    \end{align*}
    for $N$ large enough. We then conclude as in the proof of Lemma~\ref{lem:UD-for-III}, replacing Lemma~\ref{lem:cross-SMB-gap} with its WG analogue, that 
    \[
        \limsup_{N\to\infty} \frac{-\sum_{i=1}^{c_N} \ln\P\left[y_{L_{i-1,N}+1}^{L_{i-1,N}+\ell_{i,N}}\right]}{\sum_{i=1}^{c_N} \ell_{i,N}} \leq \Sc(\Q|\P).
    \]
    The proof of the complementary lower bound for the limit inferior is obtained in a similar fashion.
\end{proof}

\begin{corollary} 
\label{cor:wG}
    If, in addition to the hypotheses of Theorem~\ref{thm:sc-gen}, the measure $\P$ has the WG property and $\Q$ is ergodic, then 
    \[ 
        \lim_{N\to\infty} Q_N(y,x) = \Sc(\Q|\P)
    \]
    for $(\P\otimes\Q)$-almost every~$(x,y)$.
\end{corollary}

\section{Examples}
\label{sec:Examples}

\subsection{Hidden-Markov measures with finite hidden alphabets} 
\label{ssec:hmm}

A shortcoming of previous rigorous works on the Ziv--Merhav theorem is that measures arising from hidden-Markov models\,---\,which we are about to properly define\,---\,were not covered at a satisfying level of generality. This shortcoming is significant since hidden-Markov models are becoming increasingly popular in applications and simulations, and can display very interesting properties from a mathematical point of view. In~\cite[\S{4.4}]{BGPR}, we were essentially only able to treat them under conditions~\cite{CU03,Yo10} that guaranteed the g-measure property.

We call~$\P$ a \emph{stationary hidden-Markov measure} (HMM) if it can be represented by a tuple $(\pi, P, R)$ where $( \pi, P)$ is a stationary Markov chain on a countable set~$\mathcal{S}$, $R$ is a $(\#\mathcal{S})$-by-$(\#\cA)$ matrix whose rows are nonnegative and each sum to~1, and
\begin{equation*}
    \P[a_1^n] = \sum_{s_1^n \in \mathcal{S}^n} \pi_{s_1} R_{s_1, a_1} P_{s_1, s_2} R_{s_2, a_2}\cdots P_{s_{n-1}, s_{n}} R_{s_n, a_n}
\end{equation*}
for $n \in \nn$ and $a_1^n \in \cA^n$. One particular case of interest is when the matrix~$R$ contains only $0$s and $1$s, in which case we often speak of a function-Markov measure. In this subsection, we restrict our attention to the case where $\mathcal{S}$ is finite. 

In \cite[\S{2.2}]{BCJPPExamples}, it is shown that the class of stationary HMMs coincides with the class of stationary \emph{positive-matrix-product measures} (PMP), which are defined by a tuple $(\pi, \{M_a\}_{a \in \cA})$ where $(\pi, M)$ is a stationary Markov measure on $\Omega$ with $M \coloneqq  \sum_{a \in \cA} M_a$ and
\begin{equation*}
     \P[a_1^n] = \braket{\pi, M_{a_1} \cdots M_{a_n} \mathbf{1}}.
\end{equation*}
From now on, we only consider HMMs which admit a representation with an \emph{irreducible chain}~$( \pi, P)$.

Expressed in terms of $(\pi, P, R)$, it easy to see how one can sample sequences from such a measure\,---\,this is much more straightforward than for the measures of Section~\ref{ssec:stat-mech}\,---\,in order to perform numerical experiments. This allows us to compare the performance of the mZM estimator to the longest-match estimator, whose consistency has been established more generally~\cite{Ko98,CDEJR23w}; see Figures~\ref{fig:hmm0} and~\ref{fig:hmm-others}.

Expressed in terms of $(\pi, \{M_a\}_{a \in \cA})$, some of the decoupling properties of an irreducible HMM are more transparent: 
\begin{itemize}
\label{p:HMM-dec}
    \item Condition~\ref{it:UD} holds with $\tau = 0$ and $k_n = O(1)$;
    \item Condition~\ref{it:SLD} holds with $k_n = O(1)$;
    \item Ergodicity holds.
\end{itemize}
For a pair of such measures, a mild condition for~\ref{it:ND} (via Remark~\ref{rem:ND-decay}) was provided in~\cite[\S{4.4}]{BGPR}. Therefore, the second part of Corollary~\ref{cor:UD-tau-0} applies for pairs of irreducible HMMs with~\ref{it:ND}, and almost sure convergence to $\Sc(\Q|\P)$ holds, as in the case of a differentiable pressure at the origin, even if to the best of our knowledge, the differentiability of the cross-entropic pressure~$q$ at the origin for a pair of HMMs with finite hidden alphabet remains elusive.
It is known that not all HMMs satisfy the WG property; see e.g.~\cite[\S{2.1.2}]{BCJPPExamples}, so almost sure convergence certainly cannot be derived from Corollary~\ref{cor:wG}.

\begin{remark}
    In practice, at a given~$N$, naive implementations of the mZM and ZM algorithms will be much slower than the more well-studied longest-match length estimator as one has $c_N \sim N / \ln N$ for $N$ large under the conditions for which $Q_N \to \Sc$ as $N \to \infty$. Therefore, the complexity for $Q_N$ can be naively thought as $O( N / \ln N)$ times that of computing $\Lambda_N$. However, the foremost advantage of the mZM and ZM estimators is its faster convergence in~$N$, as seen in Figure~\ref{fig:hmm0}, which yields the most practical importance in cases where data access is limited. We do not claim any significant improvement in performance going from ZM to mZM. 
\end{remark}

\subsection{Measures arising from quantum instruments}

There is another important family of measures that shares most of the decoupling properties of HMMs. Letting~$\mathcal{H}$ denote a finite-dimensional Hilbert space, and $\mathcal{B}(\mathcal{H})$, the space of linear operators on~$\mathcal{H}$ with identity by~$\mathbf{1}$, consider a \emph{quantum measurement} described by completely positive maps $(\Phi_a)_{a\in\mathcal{A}}$ on~$\mathcal{B}(\mathcal{H})$ and a faithful initial state~$\rho$ such that $\Phi\coloneqq \sum_{a\in\mathcal{A}}\Phi_a$ satisfies $\Phi[\mathbf{1}]=\mathbf{1}$ and $\rho\circ\Phi=\rho$. As usual,~$\cA$ is assumed to be a finite set. The unraveling of $(\rho, (\Phi_a)_{a\in\mathcal{A}})$ is a shift-invariant measure $\P$ on $\mathcal{A}^{\nn}$ defined by the  marginals
\[
        \P[a_1a_2\dots a_n]
    \coloneqq \textnormal{tr}(\rho(\Phi_{a_1}\circ\dots\circ\Phi_{a_n})[\mathbf{1}]).
\]
Such measures also appear in the literature under the name ``Kusuoka measures''; see e.g.~\cite{Ku89,JOP17}.
In the case where the completely positive map $\Phi$ is irreducible, then, similar to irreducible HMMs, we have
\begin{itemize}
\label{p:RQM-dec}
    \item Condition~\ref{it:UD} holds with $\tau = 0$ and $k_n = O(1)$;
    \item Condition~\ref{it:SLD} holds with $k_n = O(1)$;
    \item Ergodicity holds;
\end{itemize}
see~\cite[\S{1}]{BCJPPExamples}. Also, for a pair of such measures, Condition~\ref{it:ND} will follow given positivity of the entropy of the second measure; see Remark~\ref{rem:ND-left-der} and Theorem~\ref{thm:left-der}. An important difference with HMMs is that, for quantum instruments, we have explicit examples of measures for which WG fails and where the cross-entropic pressure~$q$ is not differentiable at the origin, showing that Corollary~\ref{cor:UD-tau-0} does not yield results covered by either Corollary~\ref{cor:when-diff} or~\ref{cor:wG}. 

\begin{figure}
        \centering
        \includegraphics{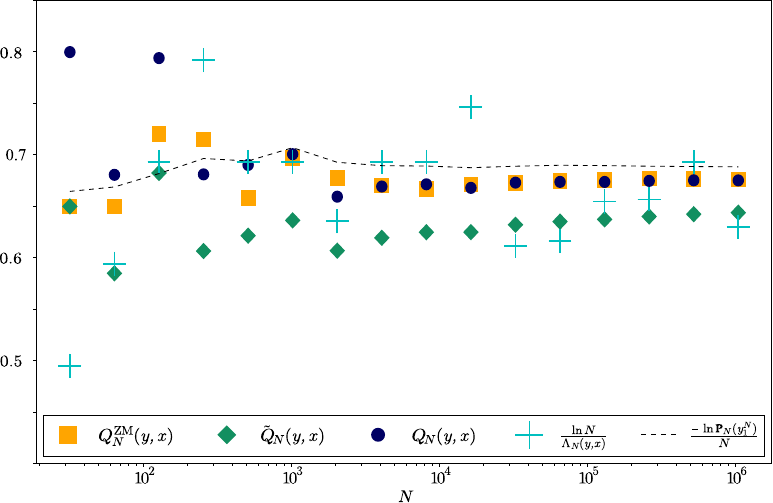} 
        \caption{The convergence of estimators to $\Sc(\Q|\P)$ is illustrated in a numerical experiment. Namely, $Q_N$ is the mZM estimator introduced herein, $\tilde{Q}_N$ is the mZM estimator without the $-c_N$ ``correction'' in the denominator, $Q_N^{\textnormal{ZM}}$ is the original ZM estimator presented in~\cite{BGPR} and~\cite{MZ93}, and~$\ln N / \Lambda_N$ is the longest-match length estimator, which has been shown to be asymptotically consistent under weaker assumptions~\cite{Ko98,CDEJR23w}. They are compared to the sequence in Lemma~\ref{lem:cross-SMB-gap}, which is computable here as we know the marginals of the measure~$\P$, which is of course not the case in practical applications. Here, both $\Q$ and $\P$ are stationary HMMs on $\{\mathsf{0},\mathsf{1}\}^\nn$.}
        \label{fig:hmm0}
\end{figure}

\begin{figure}
    \centering
    \includegraphics{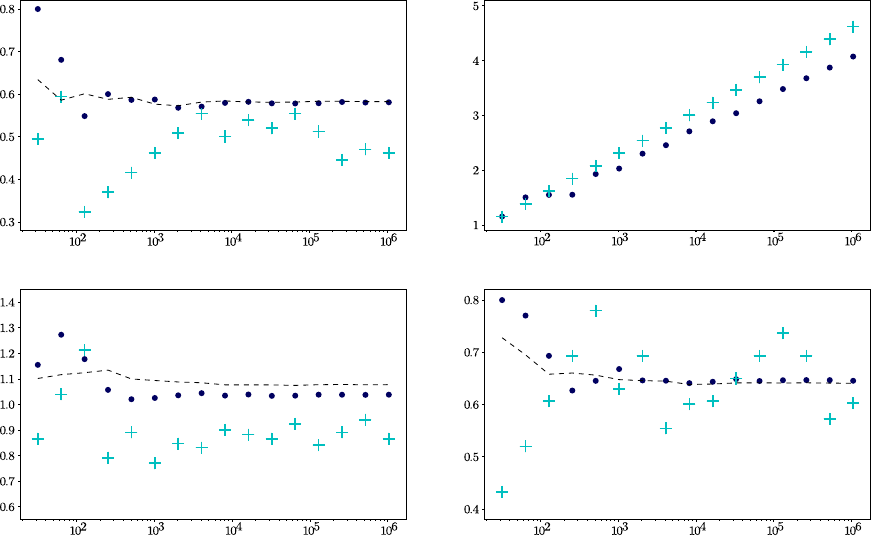}
    
    \caption{The convergence of various estimators to $\Sc(\Q|\P)$ is illustrated in more numerical experiments, subject to the same legend as Figure~\ref{fig:hmm0}: top left, both $\Q$ and $\P$ are the HMM given in Example~4.5 of~\cite[\S{4}]{BGPR}, where~\eqref{eq:add-letter} does not hold; top right, $\Q$ is fully supported Bernoulli measure on~$\{\mathsf{0},\mathsf{1},\mathsf{2}\}^\nn$  and $\P$ is a HMM on~$\Omega$ with $\supp \P \subsetneq \{\mathsf{0},\mathsf{1},\mathsf{2}\}^\nn$; bottom left, $\P$ is a Markov measure and $\Q$ a HMM on $\{\mathsf{0},\mathsf{1},\mathsf{2}\}^\nn$; bottom right, $\Q$ and $\P$ both come from the so-called ``Keep-Switch instrument''\,---\,which belongs to the family of HMMs\,---\,with parameters $(q_1, q_2) = (\tfrac 12,\tfrac 14)$; see~\cite[\S{2.1}]{BCJPPExamples}.}
    \label{fig:hmm-others}
\end{figure}

\subsection{Strong mixing conditions}

Many results in information theory beyond Markovianity are formulated in terms of strong mixing conditions such as those reviewed in~\cite{Bra05}. The strong mixing condition closest to our assumptions is the following: the shift-invariant measure~$\P$ is said to be \emph{$\psi$-mixing} if its \emph{$\psi$-mixing coefficients}, namely
\[ 
    \psi_{\P}(\tau) \coloneqq  \sup_{n,m \in \nn} \max\left\{\left|\frac{\sum_{\xi\in\cA^\tau}\P[a\xi b]}{\P[a]\P[b]}{}-1\right| : a \in \supp \P_n, b \in \supp \P_m \right\},
\]
satisfy $\psi_{\P}(\tau) \to 0$ as $\tau \to \infty$. 
Let us consider a measure~$\P$ that\,---\,to avoid discussing trivialities\,---\,satisfies
\[ 
    \delta \coloneqq  1 - \max_{a\in\cA} \P[a] > 0.
\]
If $\P$ is $\psi$-mixing, then, taking $\tau$ large enough that $(1+\psi_\P(\tau))(1-\delta) < 1$, one can show that:
\begin{itemize}
    \item Condition~\ref{it:UD} holds with this $\tau$ and with $k_n = \ln (1+\psi_\P(\tau))$;
    \item Condition~\ref{it:SLD} hold with this $\tau$ and with $k_n = -\ln (1-\psi_\P(\tau)) + \tau\ln \#\cA$;
    \item Ergodicity holds;
    \item The decay bound in Remark~\ref{rem:ND-decay} holds with $\gamma_+ > \tfrac{1}{\tau+1} \ln ((1+\psi_\P(\tau))(1-\delta))$. 
\end{itemize}
Hence, Theorems~\ref{thm:sc-gen} and~\ref{thm:left-der} apply to pairs of $\psi$-mixing measures. 
However, it should be noted that our conditions do not imply any form of mixing, as seen by considering an irreducible but periodic Markov chain.

\subsection{Long-range ferromagnetic models}
\label{ssec:stat-mech}

In this subsection, we use a particular family of models from statistical mechanics as a mean to further illustrate the applicability of some of our results. While the decoupling and decay properties are discussed in the context of more general summable interactions in~\cite{CR23,BGPR}, we focus here on long-range ferromagnetic models on~$\zz$. In this context, a translation-invariant Gibbs measure~$\P$ on~$\{-1,+1\}^\nn$ can be obtained from a limiting procedure starting from finite-volume (volume~$2n+1$) Hamiltonians 
\begin{align*}
    H_{h,J,n} : \{-1,+1\}^{\{-n,-n+1, \dotsc, n\}} &\to \rr \\
        (a_{-n},a_{-n+1}, \dotsc, a_n)  &\mapsto -\sum_{i=-n}^{n}\sum_{j = i+1}^{n} J_{|j-i|} a_i a_j - h \sum_{i=-n}^n a_i
\end{align*}
for each $n\in\nn$, where $(J_k)_{k=1}^{\infty}$ is a nonnegative, nonincreasing, summable sequence and~$h$ is a real parameter (the field strength), and from an additional positive parameter~$\beta$ (the inverse temperature). We do not describe this well-known procedure, which can be found e.g.\ in~\cite[\S{IV.2}]{Ell}.

While the WG property~\cite[\S{2}]{PS20} and~\ref{it:UD} and~\ref{it:SLD} with $\tau =0$~\cite[\S{9}]{LPS95} are guaranteed in this setup, differentiability of the pressure~$q$ at the origin could fail. 
To see this, take~$\Q$ to be a translation-invariant Gibbs measure with $$J^{\Q}_k = k^{-\frac 32}$$ and $h=0$ at $\beta > \beta_\textnormal{c}$, and $\P$ to be a translation-invariant Gibbs measure with $J^{\P}_k = 0$ and $h=1$ at the same $\beta$. Then, up to a constant factor, $\alpha \mapsto q(\alpha)$ coincides with the specific free energy for $(J^{\Q}_k)_{k=1}^\infty$ as a function of~$h$\,---\,keeping the same $\beta$. The latter is known not to be differentiable at $h = 0$; see~\cite[\S{IV.5}]{Ell}.
Because~\ref{it:ND} causes no issue here~\cite[\S{4.3}]{BGPR}, this example shows that Corollary~\ref{cor:wG} does cover cases not covered by Corollary~\ref{cor:when-diff}.

It is worth noting that for sequences $(J_k)_{k=1}^{\infty}$ that decay fast enough (or are eventually zero), the WG property can be strengthened to the Bowen--Gibbs property (i.e.\ WG with $K_n = O(1)$), so almost sure convergence can equivalently be deduced from Corollary~\ref{cor:when-diff},~\ref{cor:UD-tau-0} or~\ref{cor:wG}. 

\subsection{Hidden-Markov measures with coutable hidden alphabets}
\label{ssec:count-hid-markov-short}

In the setup of Section~\ref{ssec:hmm}, if one drops the assumption that the hidden alphabet~$\mathcal{S}$ is finite, then examples with irreducibility where our decoupling assumptions fail can easily be constructed. While a general understanding of this situation is out of reach, it allows us to construct interesting examples for exploring the overlaps between our different assumptions and results.

For example, within the family of examples discussed in~\cite[\S{A.2}]{CJPS19} and~\cite[\S{2.5}]{CR23}, different choices of parameters will push different boundaries of our results. Indeed, one can choose parameters so that Corollary~\ref{cor:UD-tau-0} yields the almost sure convergence~$Q_N \to \Sc(\Q|\P)$ even though $\Sc(\Q|\P) < D_+\overline{q}(0)$, or so that Corollary~\ref{cor:wG} yields the almost sure convergence~$Q_N \to \Sc(\Q|\P)$ even though $D_-\overline{q}(0) < \Sc(\Q|\P)$\,---\,without contradicting Theorem~\ref{thm:left-der} because the sequence~$(k_n)_{n\in\nn}$ in the~\ref{it:UD} property of~$\Q$ fails to be bounded.

\appendix

\section{A Measure-theoretic lemma}

The following lemma is an elementary estimate along the lines of the Markov or Chebyshev inequality. However, because it is used in the middle of several rather lengthy proofs, and because we could not track a convenient reference, we have opted to state and prove it separately in this appendix.
\begin{lemma}
\label{lem:Cheby-like}
    Let $\bmu$ be a finite Borel measure on~$\rr^M$. Suppose that there exist a number $K>0$ and finite, nontrivial, Borel measures $\mu_1, \dotsc \mu_M$ on~$\rr$ such that,
    as Borel measures on~$\rr^M$, 
    $$
        \bmu \leq K (\mu_1 \otimes \dotsb \otimes \mu_M).
    $$
    If we understand the exponential moments $\mathfrak{m}_i(\varkappa) \coloneqq  \int \Exp{\varkappa t_i} \dd\mu_i(t_i)$ in~$(0,\infty]$ for each $i = 1, \dotsc, M$,
    then
    \begin{align*}
        \bmu\left\{(t_i)_{i=1}^M \in \rr^M : \sum_{i=1}^M t_i \geq t \right\} \leq K\Exp{-\varkappa t} \mathfrak{m}_1(\varkappa) \dotsb \mathfrak{m}_M(\varkappa)
    \end{align*}
    for all~$\varkappa \geq 0$ and~$t \in \rr$.
    Moreover, the same estimate holds for $\sum_{i=1}^M t_i \leq t$ when $\varkappa<0$.
\end{lemma}

\begin{proof}
    Consider first the case $\varkappa \geq 0$ and let 
    $$
        A_t \coloneqq  \left\{(t_i)_{i=1}^M \in \rr^M : \sum_{i=1}^M t_i \geq t \right\}.
    $$
    If one of the exponential moments $\mathfrak{m}_i(\varkappa)$ is infinite, then there is nothing to prove. Now assuming that $\mathfrak{m}_i(\varkappa) < \infty$ for $i = 1, \dotsc, M$, we have
    \begin{align*}
        \bmu(A_t) 
            &\leq K \dotsint_{A_t} \dd\mu_1\dotsb\dd\mu_M \\
            &\leq K \dotsint_{A_t} \Exp{\varkappa\left(\sum_{i=1}^M t_i-t\right)}\dd\mu_1(t_1)\dotsb\dd\mu_M(t_M) \\
            &\leq K\Exp{-\varkappa t} \left(\int_{\rr} \Exp{\varkappa t_1}\dd\mu_1(t_1)\right) \dotsb \left(\int_{\rr}\Exp{\varkappa t_M}\dd\mu_M(t_M)\right),
    \end{align*}
    as desired. 
    With $\mathfrak{r} : \rr \to \rr$ the reflection across the origin,
    the proposed estimate for $\varkappa < 0$ and $\sum_{i=1}^M t_i \leq t$ follows from the above applied to the measures $\mu'_1 = \mu_1 \circ \mathfrak{r}^{-1}, \dotsc, \mu'_M = \mu_M \circ \mathfrak{r}^{-1}$ and $\bmu' = \bmu \circ (\mathfrak{r}\otimes\dotsb\otimes\mathfrak{r})^{-1}$ and the threshold $t' = \mathfrak{r}(t)$.
\end{proof}


\begin{thebibliography}{BBCDE08}

\bibitem[AdACG23]{AACG22}
M.~Abadi, V.~G. de~Amorim, J.-R. Chazottes, and S.~Gallo.
\newblock Return-time {${L}^{q}$}-spectrum for equilibrium states with
  potentials of summable variation.
\newblock {\em Ergodic Theor. Dyn. Syst.}, 43(8):2489--2515, 2023.

\bibitem[Bar85]{Ba85}
A.~R. Barron.
\newblock The strong ergodic theorem for densities: generalized
  {S}hannon--{M}c{M}illan--{B}reiman theorem.
\newblock {\em Ann. Probab.}, 13(4):1292--1303, 1985.

\bibitem[BBCDE08]{B+08}
C.~Basile, D.~Benedetto, E.~Caglioti, and M.~Degli~Esposti.
\newblock An example of mathematical authorship attribution.
\newblock {\em J. Math. Phys.}, 49(12):125211, 2008.

\bibitem[BCJP21]{BCJPPExamples}
T.~Benoist, N.~Cuneo, V.~Jak{\v{s}}i{\'c}, and C.-A. Pillet.
\newblock On entropy production of repeated quantum measurements~{II}.
  {E}xamples.
\newblock {\em J. Stat. Phys.}, 182(3):1--71, 2021.

\bibitem[BCL02]{BCL02}
D.~Benedetto, E.~Caglioti, and V.~Loreto.
\newblock Language trees and zipping.
\newblock {\em Phys. Rev. Lett.}, 88:048702, 2002.

\bibitem[BGPR24]{BGPR}
N.~Barnfield, R.~Grondin, G.~Pozzoli, and R.~Raqu{\'e}pas.
\newblock On the {Z}iv--{M}erhav theorem beyond {M}arkovianity {I}.
\newblock {\em Can. J. Math.}, 2024.
\newblock to appear.

\bibitem[BJPP18]{BJPP18}
T.~Benoist, V.~Jak{\v{s}}i{\'c}, Y.~Pautrat, and C.-A. Pillet.
\newblock On entropy production of repeated quantum measurements~{I}. {G}eneral
  theory.
\newblock {\em Commun. Math. Phys.}, 357(1):77--123, 2018.

\bibitem[Bra05]{Bra05}
R.~C. Bradley.
\newblock Basic properties of strong mixing conditions. {A} survey and some
  open questions.
\newblock {\em Probab. Surv.}, 2:107--144, 2005.

\bibitem[CDEJR23]{CDEJR23w}
G.~Cristadoro, M.~Degli~Esposti, V.~Jak{\v{s}}i{\'c}, and R.~Raqu{\'e}pas.
\newblock On a waiting-time result of {K}ontoyiannis: mixing or decoupling?
\newblock {\em Stoch. Proc. Appl.}, 166:104222, 2023.

\bibitem[CF05]{CF05}
D.~P. Coutinho and M.~A. Figueiredo.
\newblock Information theoretic text classification using the {Z}iv--{M}erhav
  method.
\newblock In J.~S. Marques, N.~P{\'e}rez de~la Blanca, and P.~Pina, editors,
  {\em Pattern Recognition and Image Analysis}, volume 3523 of {\em Lecture
  Notes in Computer Science}, pages 355--362. Springer, 2005.

\bibitem[CFH08]{CFH08}
Y.~Cao, D.~Feng, and W.~Huang.
\newblock The thermodynamic formalism for sub-additive potentials.
\newblock {\em Discrete Contin. Dyn. Syst.}, 20(3):639, 2008.

\bibitem[CJPS19]{CJPS19}
N.~Cuneo, V.~Jak{\v{s}}i{\'c}, C.-A. Pillet, and A.~Shirikyan.
\newblock Large deviations and fluctuation theorem for selectively decoupled
  measures on shift spaces.
\newblock {\em Rev. Math. Phys.}, 31(10):1950036, 2019.

\bibitem[CR23]{CR23}
N.~Cuneo and R.~Raqu{\'e}pas.
\newblock Large deviations of return times and related entropy estimators on
  shift spaces.
\newblock {\em arXiv preprint}, 2023.
\newblock 2306.05277 [math.PR].

\bibitem[CoTh]{CoTh}
T.~M. Cover and J.~A. Thomas.
\newblock {\em Elements of information theory}.
\newblock John Wiley \& Sons, second edition, 2006.

\bibitem[CU03]{CU03}
J.-R. Chazottes and E.~Ugalde.
\newblock Projection of {M}arkov measures may be {G}ibbsian.
\newblock {\em J. Stat. Phys.}, 111(5/6):1245--1272, 2003.

\bibitem[Ell]{Ell}
R.~S. Ellis.
\newblock {\em Entropy, large deviations, and statistical mechanics}.
\newblock Classics in mathematics. Springer-Verlag, 2006.

\bibitem[GS97]{GS97}
A.~Galves and B.~Schmitt.
\newblock Inequalities for hitting times in mixing dynamical systems.
\newblock {\em Rand. Comput. Dyn.}, 5(4):337--348, 1997.

\bibitem[J{\"O}P17]{JOP17}
A.~Johansson, A.~{\"O}berg, and M.~Pollicott.
\newblock Ergodic theory of {K}usuoka measures.
\newblock {\em J. Fractal Geom.}, 4(2):185--214, 2017.

\bibitem[Kel]{Ke98}
G.~Keller.
\newblock {\em Equilibrium states in ergodic theory}, volume~42 of {\em London
  Mathematical Society Student Texts}.
\newblock Cambridge University Press, Cambridge, 1998.

\bibitem[Kie74]{Ki74}
J.~C. Kieffer.
\newblock A simple proof of the {M}oy--{P}erez generalization of the
  {S}hannon--{M}cmillan theorem.
\newblock {\em Pacific J. Math.}, 51(1):203--206, 1974.

\bibitem[Kon98]{Ko98}
I.~Kontoyiannis.
\newblock Asymptotic recurrence and waiting times for stationary processes.
\newblock {\em J. Theor. Probab.}, 11(3):795--811, 1998.

\bibitem[KPK01]{KPK01}
O.~V. Kukushkina, A.~A. Polikarpov, and D.~V. Khmelev.
\newblock Using literal and grammatical statistics for authorship attribution.
\newblock {\em Probl. Inf. Transm.}, 37:172--184, 2001.

\bibitem[Kus89]{Ku89}
S.~Kusuoka.
\newblock {D}irichlet forms on fractals and products of random matrices.
\newblock {\em Publ. Res. Inst. Math. Sci.}, 25(4):659--680, 1989.

\bibitem[LPS95]{LPS95}
J.~T. Lewis, C.-{\'E}. Pfister, and W.~G. Sullivan.
\newblock Entropy, concentration of probability and conditional limit theorems.
\newblock {\em Markov Proc. Relat. Fields}, 1(3):319--386, 1995.

\bibitem[Moy61]{Mo61}
S.-T.~C. Moy.
\newblock Generalizations of {S}hannon--{M}c{M}illan theorem.
\newblock {\em Pacific J. Math.}, 11(2):705--714, 1961.

\bibitem[Ore85]{Or85}
S.~Orey.
\newblock On the {S}hannon--{P}erez--{M}oy theorem.
\newblock {\em Contemp. Math.}, 41:319--327, 1985.

\bibitem[PS20]{PS20}
C.-{\'E}. Pfister and W.~G. Sullivan.
\newblock Asymptotic decoupling and weak {G}ibbs measures for finite alphabet
  shift spaces.
\newblock {\em Nonlinearity}, 33(9):4799--4817, 2020.

\bibitem[Raq23]{Ra23}
R.~Raqu{\'e}pas.
\newblock A gapped generalization of {K}ingman's subadditive ergodic theorem.
\newblock {\em J. Math. Phys.}, 64(6):06270, 2023.

\bibitem[Shi93]{Sh93}
P.~C. Shields.
\newblock Waiting times: positive and negative results on the {W}yner--{Z}iv
  problem.
\newblock {\em J. Theor. Probab.}, 6(3):499--519, 1993.

\bibitem[Sim]{Sim}
B.~Simon.
\newblock {\em The Statistical Mechanics of Lattice Gases}, volume~1.
\newblock Princeton University Press, 1993.

\bibitem[Wa01]{Wa01}
P.~Walters.
\newblock Convergence of the {R}uelle operator for a function satisfying
  {B}owen's condition.
\newblock {\em Trans. Amer. Math. Soc.}, 353(1):327--347, 2001.

\bibitem[WZ89]{WZ89}
A.~D. Wyner and J.~Ziv.
\newblock Some asymptotic properties of the entropy of a stationary ergodic
  data source with applications to data compression.
\newblock {\em IEEE Int. Symp. Inf. Theory}, 35(6):1250--1258, 1989.

\bibitem[Yoo10]{Yo10}
J.~Yoo.
\newblock On factor maps that send {M}arkov measures to {G}ibbs measures.
\newblock {\em J. Stat. Phys.}, 141(6):1055--1070, 2010.

\bibitem[ZL78]{LZ78}
J.~Ziv and A.~Lempel.
\newblock Compression of individual sequences via variable-rate coding.
\newblock {\em IEEE Trans. Inf. Theory}, 24(5):530--536, 1978.

\bibitem[ZM93]{MZ93}
J.~Ziv and N.~Merhav.
\newblock A measure of relative entropy between individual sequences with
  application to universal classification.
\newblock {\em IEEE Trans. Inf. Theory}, 39(4):1270--1279, 1993.

\end{thebibliography}
\end{document}